\documentclass[11pt,reqno]{article}
\usepackage{amsmath,amsfonts,amsthm}
\usepackage[margin=1in]{geometry}
\usepackage[normalem]{ulem}
\usepackage{ amssymb }

\usepackage[hyphens]{url}
\usepackage{hyperref}
\usepackage{graphicx}
\usepackage{verbatim}
\usepackage{enumitem} 
\usepackage{color}
\usepackage{tikz}
\usetikzlibrary{decorations.fractals}
\usetikzlibrary{patterns,arrows,shapes,decorations.pathreplacing}
\usepackage{bm}
\usepackage{xcolor}

\numberwithin{equation}{section}

\newtheorem{theorem}{Theorem}[section]
\newtheorem{assumption}{Assumption}[section]
\newtheorem{corollary}{Corollary}[section]
\newtheorem{lemma}{Lemma}[section]
\newtheorem{proposition}{Proposition}[section]

\newtheorem{definition}{Definition}[section]
\newtheorem{remark}{Remark}[section]
\newtheorem{example}{Example}[section]

\newcommand{\lemref}{Lemma~\ref}

\newcommand{\propref}{Proposition~\ref}
\newcommand{\thmref}{Theorem~\ref}

\renewcommand{\P}{\mathbb{P}}

\newcommand{\R}{\mathbb{R}}
\newcommand{\E}{\mathbb{E}}
\newcommand{\cE}{\mathcal{E}}
\newcommand{\N}{\mathbb{N}}
\newcommand{\F}{\mathcal{F}}

\newcommand{\X}{\mathbb{X}}

\newcommand{\B}{\mathcal{B}}

\newcommand{\Z}{\mathbb{Z}}
\newcommand{\T}{\mathcal{T}}

\newcommand{\eps}{\varepsilon}

\newcommand{\qq}{\quad\text{and}\quad}

\newcommand{\nada}[1]{}
\definecolor{gb}{rgb}{0, 0.2, 0.8}

\title{A Time-Inconsistent Dynkin Game: from Intra-personal to Inter-personal Equilibria}
\author{Yu-Jui Huang\thanks{
University of Colorado, Department of Applied Mathematics, Boulder, CO 80309-0526, USA, email: \texttt{yujui.huang@colorado.edu}. Partially supported by National Science Foundation (DMS-1715439) and the University of Colorado (11003573).}
 \and Zhou Zhou\thanks{
University of Sydney, School of Mathematics and Statistics, NSW 2006, Australia, email: \texttt{zhou.zhou@sydney.edu.au}.}
}
\date{\today}
\begin{document}
\maketitle

\begin{abstract}
This paper studies a nonzero-sum Dynkin game in discrete time under non-exponential discounting. For both players, there are two levels of game-theoretic reasoning intertwined. First, each player looks for an {\it intra-personal} equilibrium among her current and future selves, so as to resolve time inconsistency triggered by non-exponential discounting. Next, given the other player's chosen stopping policy, each player selects a best response among her intra-personal equilibria. A resulting {\it inter-personal} equilibrium is then a Nash equilibrium between the two players, each of whom employs her best intra-personal equilibrium with respect to the other player's stopping policy. Under appropriate conditions, we show that an inter-personal equilibrium exists, based on concrete iterative procedures along with Zorn's lemma. To illustrate our theoretic results, we investigate a two-player real options valuation problem: two firms negotiate a deal of cooperation to initiate a project jointly. By deriving inter-personal equilibria explicitly, we find that coercive power in negotiation depends crucially on the impatience levels of the two firms.  
\end{abstract}

\smallskip



\section{Introduction}
In dynamic optimization, {\it time inconsistency} is the self-conflicting situation where the same agent at different times (i.e. the current and future selves) cannot agree on a ``dynamically optimal strategy'' that is good for the entire planning horizon. A long-standing approach to resolving time inconsistency is Strotz' {\it consistent planning} \cite{Strotz55}: An agent should take her future selves' disobedience into account, so as to find a strategy that none of her future selves will have an incentive to deviate from. Essentially, such a strategy is an {\it intra-personal} equilibrium---an equilibrium established internally within the agent, among her current and future selves. 

The investigation of intra-personal equilibria, particularly their mathematical definitions and characterizations, has been the main focus of the literature on time inconsistency. This includes the classical framework in discrete time that relies on a straightforward backward sequential optimization detailed in \cite{Pollak68}, as well as the more recent development in continuous time that employs the spike variation technique introduced in Ekeland and Lazrak \cite{EL06}. The latter has led to vibrant research on time-inconsistent stochastic control, including \cite{EP08, EMP12, BMZ14, BKM17, Yong12}, among many others. 
Lately, marked progress has been made for time-inconsistent optimal stopping, along two different paths. One is to extend the spike variation technique from stochastic control to optimal stopping, as carried out in \cite{EWZ18, CL18, CL20}. The other path is the iterative approach developed in \cite{HN18, HNZ20, HY19}, which circumvents spike variations via a fixed-point perspective. Let us also mention the recent work \cite{https://doi.org/10.1111/mafi.12293} which builds a connection between different concepts of equilibria in these two paths.

A natural question follows all the developments: How does the intra-personal reconciliation within one single agent integrate into the interaction among multiple (non-cooperative) agents? Intuitively, there should be two levels of game-theoretic reasoning---the {\it inner} level where each agent looks for time-consistent strategies her future selves will actually follow, and the {\it outer} level where each agent chooses her best strategy (among time-consistent ones) in response to other agents' strategies. A resulting {\it inter-personal} equilibrium should then be a Nash equilibrium among all the agents, each of whom is restricted to choose time-consistent strategies (i.e. her intra-personal equilibria). To the best of our knowledge, such inter-personal equilibria built from intra-personal ones have not been properly formulated and studied in the literature. 
The crucial question is {\it whether} and {\it how} different agents' respective intra-personal equilibria can ultimately forge an inter-personal equilibrium among all agents. This paper will shed new light on this through a time-inconsistent Dynkin game. 

A Dynkin game involves two players interacting through their stopping strategies. 
The zero-sum version of the game, introduced in Dynkin \cite{Dynkin69}, has been substantially studied along various directions, including \cite{Dynkin69, Neveu-book-75} (discrete time), \cite{Bismut77, LM84, Morimoto84} (continuous time), \cite{Yasuda85, RSV01, TV02} (randomized strategies), \cite{CK96} (non-Markovian settings), and \cite{BY17} (model uncertainty), among others. Many of the studies not only show that a Nash equilibrium between the two players exists, but provide concrete constructions. By contrast, the nonzero-sum version of the game has received relatively less attention; see the early investigations \cite{Morimoto86, Ohtsubo87, Nagai87} and more recent ones \cite{HZ09, LS13, DFM18}, among others. Remarkably, all the developments above assume that the two players optimize their expected payoff/cost under exponential discounting (including the case of no discounting), which readily ensures time consistency. 

In this paper, we consider a nonzero-sum Dynkin game where the state process $X$ is a discrete-time strong Markov process taking values in a Polish space $\X$. Each player chooses to stop at the first entrance time of some Borel subset $S$ of $\X$, which will be called a stopping policy. The Dynkin game is in general time-inconsistent, as we allow the two players to take general discount functions that satisfy only a log sub-additive condition, i.e. \eqref{DI} below. This condition captures {\it decreasing impatience}, a widely observed feature of empirical discounting, and readily covers numerous non-exponential discount functions in behavioral economics; see the discussion below \eqref{DI}. 

As time inconsistency arises under non-exponential discounting, each player, when given the other's chosen stopping policy $T$, needs to find accordingly an intra-personal equilibrium $S$ among her current and future selves. Following the fixed-point approach in Huang and Nguyen-Huu \cite{HN18}, we define each player's intra-personal equilibrium as a fixed point of an operator, which encodes the aforementioned  inner level of game-theoretic reasoning (Definition~\ref{def:E}). To achieve an inter-personal equilibrium, a minimal requirement is that each player should attain her inner-level equilibrium simultaneously---that is, the following situation should materialize: $S$ is Player 1's intra-personal equilibrium given Player 2's stopping policy $T$, and $T$ is Player 2's intra-personal equilibrium given Player 1's stopping policy $S$. In this case, we say $(S,T)$ is a {\it soft} inter-personal equilibrium (Definition~\ref{def:soft}). To further refine this ``soft'' definition, we note that each player, when following the Nash equilibrium idea, should {\it not} be satisfied with an arbitrary intra-personal equilibrium, but aim at the {\it best} one under an appropriate optimality criterion. Reminiscent of the ``optimal equilibrium'' concept proposed in Huang and Zhou \cite{HZ19, HZ20}, we say that an intra-personal equilibrium is optimal if it generates larger values than any other intra-personal equilibrium, for not only the current but all future selves (Definition~\ref{def:optimal E}). 
This immediately brings about a stronger notion for an inter-personal equilibrium: $(S,T)$ is said to be a {\it sharp} inter-personal equilibrium if each player attains her {\it best} inner-level equilibrium simultaneously, i.e. $S$ is Player 1's {\it optimal} intra-personal equilibrium given Player 2's stopping policy $T$, and $T$ is Player 2's {\it optimal} intra-personal equilibrium given Player 1's stopping policy $S$ (Definition~\ref{def:sharp}).

The focus of this paper is to establish the existence of inter-personal equilibria, soft and sharp, through concrete iterative procedures. First, we develop for each player an {\it individual} {iterative procedure}, i.e. \eqref{S_n} below, that directly leads to her optimal intra-personal equilibrium (Theorem~\ref{t1}). This procedure can be viewed as an improvement to those in \cite{HN18, HNZ20}, which lead to intra-personal equilibria but not necessarily the optimal ones. Next, we devise an {\it alternating} iterative procedure, i.e. \eqref{alternating} below, in which the two players {take turns} to perform the individual iterative procedure repetitively. In each iteration, one player, given the other's stopping policy determined in the previous iteration, performs the individual iterative procedure and then updates her policy to the optimal intra-personal equilibrium obtained; see Section~\ref{subsec:problem} for details. Under appropriate conditions, this alternating iterative procedure converges and the limit, denoted by $(S_\infty, T_\infty)$, is guaranteed a soft inter-personal equilibrium (Theorem~\ref{t2}). While it is tempting to believe that $(S_\infty, T_\infty)$ is in fact sharp, in view of its structure revealed in Theorem~\ref{t2}, this is generally not the case: We demonstrate explicitly that $(S_\infty, T_\infty)$ is sharp in Example~\ref{eg1}, but only soft in the slightly modified Example~\ref{eg2}. In other words, the general existence of sharp inter-personal equilibria is still in question. Assuming additionally that the state process $X$ has transition densities, we are able to upgrade the construction of $(S_\infty, T_\infty)$ and apply Zorn's lemma appropriately, which yields the desired result that a sharp inter-personal equilibrium must exist (Theorem~\ref{t3}). 

It is worth noting that Theorems~\ref{t2} and \ref{t3} hinge on a supermartingale condition, i.e. \eqref{supermartingale} below. As shown in Section~\ref{subsec:supermartingale}, when the supermartingale condition fails, there may exist no inter-personal equilibrium, either soft or sharp; see Proposition~\ref{prop:no weak E} particularly. Let us point out that similar supermartingale conditions were also imposed in some studies on classical (time-consistent) nonzero-sum Dynkin games (e.g. \cite{Morimoto86, Ohtsubo87}) to facilitate the existence of a Nash equilibrium.

As an application, we study the negotiation between two firms (or countries) in Section~\ref{sec:application}. Suppose that each firm intends to coerce the other into unfavorable terms so as to obtain a larger payoff. A firm either waits until the other gives in and takes the larger payoff (i.e. its coercion works), or gives in to the other and accepts the unfavorable terms (i.e. its coercion fails). By computing explicitly the sharp inter-personal equilibrium between the two firms (Propositions~\ref{prop:1<2} and \ref{prop:1>2}, Corollary~\ref{coro:1>2}), we find that whether coercion in negotiation works depends on the impatience levels of the two firms: If a firm is less impatient than the other, its coercion always works; on the other hand, if a firm is significantly more impatient than the other, its coercion must fail. See particularly the discussion below Corollary~\ref{coro:1>2} for details.

The rest of the paper is organized as follows. Section~\ref{sec:model} introduces the model setup, formulates intra- and inter-personal equilibria, and collects preliminary results. Section~\ref{sec:one-player} develops an individual iterative procedure that directly yields a player's optimal intra-personal equilibrium. Under a supermartingale condition, the monotonicity of this procedure is also established. Section~\ref{sec:existence} devises an alternating iterative procedure, from which we prove the existence of soft and sharp inter-personal equilibria. Examples are presented to demonstrate the alternating iterative procedure and the necessity of the supermartingale condition. Finally, Section~\ref{sec:application} applies our analysis to the negotiation between two firms, relating coercive power to impatience level.


\section{The Model and Preliminaries}\label{sec:model}
Let $\Z_+:=\{0,1,2,...\}$ and consider a time-homogeneous strong Markov process $X = (X_t)_{t\in\Z_+}$ taking values in a Polish space $\X$. We denote by $\B$ the Borel $\sigma$-algebra of $\X$. On the path space $\Omega$, the set of all functions mapping $\Z_+$ to $\X$, let $(\F_t)_{t\in\Z_+}$ be the filtration generated by $X$ and $\T$ be the set of all $(\F_t)_{t\in\Z_+}$-stopping times. In addition, we consider $\F_\infty:= \bigcup_{t\in \Z_+}\F_t$. For any $x\in\X$, we denote by $X^x$ the process $X$ with initial value $X_0=x$, by $\P_x$ the probability measure on $(\Omega,\F_\infty)$ generated by $X^x$ (i.e. the law of $(X^x_t)_{t\in\Z_+}$), and by $\E_x$ the expectation under $\P_x$. 

Consider a nonzero-sum Dynkin game where the two players maximize their respective expected payoffs, determined jointly by their stopping strategies. Specifically, for $i\in\{1,2\}$, given the stopping time $\sigma\in\T$ chosen by the  other player, Player $i$ at the current state $x\in\X$ selects a stopping time $\tau\in\T$ to maximize her expected discounted payoff
\begin{equation}\label{J}
J_i(x,\tau,\sigma):=\E_x[F_i(\tau,\sigma)], 
\end{equation}
where
\begin{equation}\label{F}
F_i(\tau,\sigma):=\delta_i(\tau)f_i(X_\tau)1_{\{\tau<\sigma\}}+\delta_i(\sigma)g_i(X_\sigma)1_{\{\tau>\sigma\}}+\delta_i(\tau)h_i(X_\tau)1_{\{\tau=\sigma\}},\quad \forall\tau,\sigma\in\T.
\end{equation}
Here, $\delta_i:\Z_+\to [0,1]$ is Player $i$'s discount function, assumed to be strictly decreasing with $\delta(0)=1$, and $f_i, g_i, h_i:\X\to\R_+$ are Player $i$'s payoff functions, assumed to be Borel measurable. Note that we allow $\tau, \sigma\in\T$ to take the value $+\infty$. For any $\omega\in\{\tau=\sigma=+\infty\}$, we simply define $F_i(\tau,\sigma)(\omega):= \limsup_{t\to\infty} \delta_i(t)h_i(X_t(\omega))$. To ensure that $J_i(x,\tau,\sigma)$ in \eqref{J} is well-defined, we will impose throughout the paper 
\begin{equation}\label{dominated}
\E_x\bigg[\sup_{t\in\N}\delta_i(t)\big(f_i(X_t)+g_i(X_t)+h_i(X_t)\big)\bigg]<\infty\quad \forall x\in\X.
\end{equation}

As mentioned in Introduction, the vast literature on Dynkin games mostly assumes exponential discounting, i.e. $\delta_i(t)=e^{-\beta_i t}$ for some $\beta_i>0$. Empirical studies (e.g. \cite{Thaler81, LT89}), on the other hand, have found that individuals do not normally discount exponentially. In this paper, the only standing assumption on $\delta_i$, $i\in\{1,2\}$, is
\begin{equation}\label{DI}
\delta_i(s)\delta_i(t)\le \delta_i(s+t),\quad \forall s,t\in\Z_+.
\end{equation}
This particularly captures {\it decreasing impatience}, a widely observed feature of empirical discounting. Numerous non-exponential discount functions in behavioral economics, such as hyperbolic, generalized hyperbolic, and pseudo-exponential discount functions, readily satisfy \eqref{DI}; see the discussion below \cite[Assumption 3.12]{HN18} for details. 

In a one-player stopping problem, it is well-understood that non-exponential discounting induces time inconsistency: An optimal stopping strategy derived at the current state $x\in\X$ may no longer be optimal at a subsequent state $y\neq x$. In other words, the current and future selves cannot agree on a ``dynamically optimal stopping strategy'' that is good for the entire planning horizon; see e.g. \cite[Section 2.2]{HN18} for an explicit demonstration. Strotz' {\it consistent planning} \cite{Strotz55} is a long-standing approach to resolving time inconsistency: Knowing that her future selves may overturn her current plan, an agent selects the best present action taking the future disobedience as a constraint; the resulting strategy is a (subgame perfect) Nash equilibrium from which no future self has an incentive to deviate. 

In our Dynkin game, thanks to the time-homogeneous Markovian setup, we assume that each player decides to stop or to continue depending on her current state $x\in\X$. That is, each player stops at the first entrance time of some $S\in\B$, defined by 
\[
\rho_S:=\inf\{t\geq 0:\ X_t\in S\}. 
\]
For convenience, we will often call $S\in\B$ a {\it stopping policy}. This corresponds to a ``pure strategy'' in economic terms.\footnote{See Remark~\ref{rem:randomized} for discussions on the use of pure and randomized strategies.} For $i\in\{1,2\}$, given the other player's stopping policy  $T\in\B$, Player $i$ is faced with time inconsistency among her current and future selves (as explained above), and needs to find an equilibrium stopping policy at the {\it intra-personal} level. Following \cite[Section 2.1]{HZ20} (or \cite[Section 3.1]{HN18}), Strotz' consistent planning boils down to the current self's game-theoretic reasoning: ``Given that my future selves will follow the policy $S\in\B$, what is the best policy today in response to that?'' The best policy is determined by comparing the payoff of immediate stopping $J_i(x,0,\rho_T)$ and the payoff of continuation $J_i(x,\rho^+_S,\rho_T)$, where 
\[
\rho^+_S:=\inf\{t>0:\ X_t\in S\}
\]
is the first hitting time to $S$. This leads to the following stopping policy
\begin{align}\label{Theta}
\Theta_i^T(S):=&\{x\in S: J_i(x,0,\rho_T)\geq J_i(x,\rho^+_S,\rho_T)\}\cup\{x\notin S: J_i(x,0,\rho_T)>J_i(x,\rho^+_S,\rho_T)\}\in \B.
\end{align}
We can consider $\Theta^T_i:\B\to\B$ as an {\it improving} operator for Player $i$: Given the other player's stopping policy $T\in\B$, $\Theta^T_i$ improves the present policy $S\in\B$ of Player $i$ to $\Theta^T(S)\in\B$. 

For {each} player, we define an equilibrium at the {\it intra-personal} level (i.e. among the player's current and future selves) in the same spirit as \cite[Definition 2.2]{HZ20} and \cite[Definition 3.7]{HN18}.

\begin{definition}\label{def:E}
For $i\in\{1,2\}$, $S\in\B$ is called Player $i$'s intra-personal equilibrium w.r.t. (with respect to) $T\in\B$ if $\Theta_i^T(S)=S$. We denote by $\cE_i^T$ the set of all Player $i$'s intra-personal equilibria w.r.t. $T\in\B$. 
\end{definition}

\begin{remark}
The above fixed-point definition of an intra-personal equilibrium was introduced in \cite{HN18} and followed by \cite{HNZ20, HZ20, HY19}, among others. 
Note that there is a slightly different formulation in \cite{HZ19}: If we follow \cite{HZ19}, particularly (2.5) therein, $\Theta^T_i(S)$ in \eqref{Theta} needs to be modified as
\begin{align}\label{Theta'}
\bar \Theta_i^T(S):=\{x\in \X: J_i(x,0,\rho_T)\geq J_i(x,\rho^+_S,\rho_T)\}.
\end{align}
Observe from \eqref{Theta} and \eqref{Theta'} that the equilibrium condition ``$\Theta_i^T(S)=S$'' in Definition~\ref{def:E} is slightly weaker than ``$\bar \Theta_i^T(S)=S$'' as in \cite[Definition 2.3]{HZ19}. 
This paper uses the slightly weaker definition because it conforms more closely to the Nash equilibrium idea---one deviates to a new policy only when it is strictly better than the current one; see the explanations at the beginning of \cite[p.7]{HN18}. Moreover, the weaker definition facilitates the search for intra-personal equilibria, as it allows for the explicit construction in Proposition~\ref{l3} below.  
\end{remark}

Based on Definition~\ref{def:E}, we introduce the first kind of equilibria at the {\it inter-personal} level (i.e. between the two players)---the {\it soft} inter-personal equilibria.

\begin{definition}\label{def:soft}
We say $(S,T)\in\B\times\B$ is a soft inter-personal equilibrium (for the Dynkin game) if $S\in \cE_1^T$ and $T\in\cE_2^S$ (i.e. $\Theta_1^T(S)=S$ and $\Theta_2^S(T)=T$). 
We denote by $\cE$ the set of all soft inter-personal equilibria.  
\end{definition}

Essentially, $(S,T)\in\cE$ means that each player {\it simultaneously} attains an equilibrium at the intra-personal level, 
given the other's stopping policy: $S$ is Player $1$'s intra-personal equilibrium w.r.t. Player 2's policy $T$, and $T$ is Player $2$'s intra-personal equilibrium w.r.t. Player 1's policy $S$.

As emphasized in \cite{HZ19, HZ20}, Strotz' consistent planning is a {\it two-phase} procedure: An agent first determines the strategies that she will actually follow over time ({\it Phase I}), and then chooses the best one among them ({\it Phase II}). In our Dynkin game, {\it Phase I} amounts to each player finding her intra-personal equilibria (w.r.t. the other player's stopping policy); {\it Phase II} is then the search for an {\it optimal} intra-personal equilibrium, defined as below.

\begin{definition}\label{def:optimal E}
For $i\in\{1,2\}$ and $T\in\B$, the value function associated with $S\in\cE_i^T$ is defined by 
\[
U_i^T(x, S):=J_i(x,0,\rho_T)\vee J_i(x,\rho^+_S,\rho_T),\quad  x\in\X. 
\]
We say $S\in\cE_i^T$ is Player $i$'s optimal intra-personal equilibrium w.r.t. $T\in\B$ if for any $R\in\cE_i^T$,
\[
U_i^T(x,S)\ge U_i^T(x,R)\quad \hbox{for all}\ x\in\X. 
\]
We denote by $\widehat{\cE}_i^T$ the set of all Player $i$'s optimal intra-personal equilibria w.r.t. $T$. 
\end{definition}

\begin{remark}
Thanks to $S\in\cE_i^T$, $U_i^T(x,S)$ defined above coincides with $J_i(x,\rho_S,\rho_T)$. Indeed, by $\Theta_i^T(S) = S$ (due to $S\in\cE_i^T$) and \eqref{Theta}, $J_i(x,\rho_S,\rho_T) = J_i(x,0,\rho_T)\ge J_i(x,\rho^+_S,\rho_T)$ for $x\in S$ and $J_i(x,\rho_S,\rho_T) =  J_i(x,\rho^+_S,\rho_T)\ge J_i(x,0,\rho_T)$ for $x\notin S$. That is, $J_i(x,\rho_S,\rho_T) = J_i(x,0,\rho_T)\vee J_i(x,\rho^+_S,\rho_T) = U_i^T(x,S)$ for all $x\in\X$. 
\end{remark}

Definition~\ref{def:optimal E} follows the ``optimal equilibrium'' notion introduced in \cite{HZ19}. It is a rather strong optimality criterion, as it requires a (subgame perfect Nash) equilibrium to dominate any other equilibrium on the entire state space---a rare occurrence in game theory. Nonetheless, for the one-player optimal stopping problem under non-exponential discounting, as long as the discount function satisfies \eqref{DI}, the existence of an optimal equilibrium has been established first in discrete time \cite{HZ19} and then in continuous time, including \cite{HZ20, HW20} (diffusion models) and \cite{https://doi.org/10.1111/mafi.12293} (continuous-time Markov chain models).

Based on Definition~\ref{def:optimal E}, we introduce the second kind of equilibria at the {\it inter-personal} level (i.e. between the two players)---the {\it sharp} inter-personal equilibria.

\begin{definition}\label{def:sharp}
We say $(S,T)\in\B\times\B$ is a sharp inter-personal equilibrium (for the Dynkin game) if $S\in\widehat{\cE}_1^T$ and $T\in \widehat{\cE}_2^S$. We denote by $\widehat \cE$ the set of all sharp inter-personal equilibria.  
\end{definition}

A sharp inter-personal equilibrium, 
compared with a soft one in Definition~\ref{def:soft}, conforms to the Nash equilibrium concept more closely. 
Given the other player's policy, what a player aims at should {\it not} be an arbitrary agreement among her current and future selves (as is stipulated in Definition~\ref{def:soft}), {\it but} the agreement that is best-rewarding---the one that generates the largest possible value for every incarnation of herself in time.  
In other words, our time-inconsistent Dynkin game involves {\it two levels} of game-theoretic reasoning. Player 1 wants to find the best response to Player 2's policy {\it at the inter-personal level}, while maintaining an agreement among her current and future selves {\it at the intra-personal level}; Player 2 does the same in response to Player 1's policy. In the end, each player chooses an optimal intra-personal equilibrium w.r.t. the other player's policy, leading to a sharp inter-personal equilibrium for the Dynkin game.

It is worth noting that our definition of a sharp inter-personal equilibrium covers, as a special case, the standard Nash equilibrium in a time-consistent Dynkin game. 

\begin{remark}
In the time-consistent case of exponential discounting, a Nash equilibrium for the Dynkin game is defined as a tuple of stopping times $(\hat \tau,\hat \sigma)$ such that $\hat \tau$ is Player 1's optimal stopping time given that Player 2 employs $\hat \sigma$, while $\hat\sigma$, at the same time, is Player 2's optimal stopping time given that Player 1 employs $\hat \tau$. In a time-homogeneous setting, let $\hat S$ (resp. $\hat T$) denote the stopping region associated with $\hat\tau$ (resp. $\hat\sigma$). In view of \cite[Proposition 3.11]{HN18}, $\hat S$ is readily Player 1's intra-personal equilibrium w.r.t. $\hat T$. Moreover, by the argument in \cite[Remark 2.12]{HW20}, $\hat S$ is in fact Player 1's {\it optimal} intra-personal equilibria w.r.t. $\hat T$---namely, $\hat S\in \widehat{\mathcal E}^{\hat T}_1$. By the same token, we have  $\hat T\in \widehat{\mathcal{E}}^{\hat S}_2$. It then follows that $(\hat S,\hat T)$ is a sharp inter-personal equilibrium.  

That is to say, in the classical time-consistent case, a Nash equilibrium $(\hat \tau,\hat \sigma)$ in a time-homogeneous model is automatically a sharp inter-personal equilibrium, once we re-state $(\hat \tau,\hat \sigma)$ using their respective stopping regions. 
\end{remark}

\subsection{Problem Formulation}\label{subsec:problem}
This paper aims to establish the {existence} of soft and sharp inter-personal equilibria, using concrete iterative procedures.
Although the existence and construction of each player's intra-personal equilibria is well-understood (based on the one-player results  \cite{HN18, HZ19}), it is unclear whether the two players' respective intra-personal equilibria can be coordinated properly to form an inter-personal equilibrium, either soft or sharp. 


We will tackle this in two steps. First, we look into the one-player problem more closely, developing for each player an {\it individual} {iterative procedure} that directly brings about her optimal intra-personal equilibrium (Theorem~\ref{t1}). 
Next, we devise an {\it alternating} iterative procedure in which the two players {\it take turns} to perform the individual iterative procedure:
\begin{itemize}
\item [1.] With respect to Player 1's initial policy $S_0\in\B$, Player 2 performs the individual iterative procedure to get an optimal intra-personal equilibrium $T_0\in\B$.
\item [2.] With respect to Player 2's policy $T_0\in\B$, Player 1 performs the individual iterative procedure to get an optimal intra-personal equilibrium $S_1\in\B$.
\item [3.] With respect to Player 1's policy $S_1\in\B$, Player 2 performs the individual iterative procedure to get an optimal intra-personal equilibrium $T_1\in\B$.
\item [\vdots]
\end{itemize}
The hope is that this alternating iterative procedure will ultimately converge, with the limit $(S_\infty, T_\infty)$ being a soft, or even sharp, inter-personal equilibrium. This will be investigated in detail in Section~\ref{sec:existence}, with affirmative results established in Theorems~\ref{t2} and \ref{t3}.

\subsection{Preliminaries}
We collect two technical results that will be useful throughout the paper. The first one concerns the convergence of first entrance and hitting times.

\begin{lemma}\label{l1}
Let $(S_n)_{n\in\N}$ be a monotone sequence in $\B$. For any $\omega\in\Omega$, there exists $N\in\N$ such that $\rho_{S_n}(\omega)=\rho_{S_\infty}(\omega)$ for all $n\ge N$, where 
\begin{equation}
S_\infty := 
\begin{cases}
\bigcup_{n\in\N} S_n,\quad \hbox{if}\ (S_n)\ \hbox{is nondecreasing},\\
\bigcap_{n\in\N} S_n,\quad \hbox{if}\ (S_n)\ \hbox{is nonincreasing}. 
\end{cases}
\end{equation}
The same result holds with $\rho$ replaced by $\rho^+$. 
\end{lemma}

\begin{proof}
Fix $\omega\in\Omega$. If $(S_n)$ is nondecreasing, set $t:=\rho_{S_\infty}(\omega)$. Without loss of generality, assume  $t<\infty$. As $X_t(\omega)\in S_\infty=\bigcup_{n\in\N} S_n$, there exists $N\in\N$ such that $X_t(\omega)\in S_n$ for all $n\ge N$. Hence, $\rho_{S_n}(\omega)\leq t$ for all $n\ge N$. If there exists $n^*\ge N$ such that $\rho_{S_{n^*}}(\omega)< t$, then $\rho_{S_\infty}(\omega)\le \rho_{S_{n^*}}(\omega)<t$, a contradiction. We thus conclude $\rho_{S_n}(\omega)= t = \rho_{S_\infty}(\omega)$ for all $n\ge N$. 
On the other hand, if $(S_n)$ is nonincreasing, set $t:=\lim_{n\to\infty}\rho_{S_n}(\omega)$. Without loss of generality, assume $t<\infty$. Then, there exists $N\in\N$ such that $\rho_{S_n} (\omega)=t$ for all $n\ge N$. Hence, $X_{t}(\omega)\in S_n$ for all $n\ge N$ and thus $X_t(\omega)\in S_\infty=\bigcap_{n\in\N} S_n$. This implies $\rho_{S_\infty}(\omega)\le t$. Since $\rho_{S_\infty}(\omega)\ge t$ by definition, we conclude $\rho_{S_\infty}(\omega)= t=\rho_{S_n}(\omega)$ for all $n\ge N$. 
 The same arguments as above hold with $\rho$ replaced by $\rho^+$. 
\end{proof}

The next result states that {\it any} stopping policy containing an intra-personal equilibrium $R$ must be dominated by $R$. This kind of result was first established for one-player stopping problems in \cite[Lemma 3.1]{HZ20}, and is now extended to a Dynkin game setting.  

\begin{lemma}\label{l5}
Fix $i\in\{1,2\}$ and assume $h_i\leq g_i$. Then, for any $R$, $S\in\B$ with $R\subseteq S$ and $R\in\cE_i^T$ for some $T\in\B$, 
\[
J_i(x,\rho^+_R,\rho_T)\geq J_i(x,\rho^+_S,\rho_T)\quad  \hbox{for all}\ x\in\X.
\]
\end{lemma}

\begin{proof}
Consider $A:=\{\omega\in\Omega: \rho^+_S=\rho_T<\rho^+_R\}$ and $B:=\{\omega\in\Omega: \rho^+_S<\rho_T\wedge\rho^+_R\}$. For any $x\in\X$, by \eqref{J} and \eqref{F}, 
\begin{align}\label{1_A,B}
J_i(x,\rho^+_R,\rho_T)-J_i(x,\rho^+_S,\rho_T)=\E_x\left[(1_A+1_B)(F_i(\rho^+_R,\rho_T)-F_i(\rho^+_S,\rho_T)\right].
\end{align}
By the assumption $g_i\ge h_i$, 
\begin{equation}\label{1_A}
\E_x\left[1_A(F_i(\rho^+_R,\rho_T)-F_i(\rho^+_S,\rho_T))\right]=\E_x[1_A(\delta_i(\rho_T)g_i(X_{\rho_T})-\delta_i(\rho_T)h_i(X_{\rho_T}))]\geq 0.
\end{equation}
On the other hand, 
\begin{align}\label{1_B}
\E_x\left[1_B(F_i(\rho^+_R,\rho_T)-F_i(\rho^+_S,\rho_T))\right]&=\E_x\left[1_B\left(\E_x\left[F_i(\rho^+_R,\rho_T)\Big|\mathcal{F}_{\rho^+_S}\right]-F_i(\rho^+_S,\rho_T)\right)\right].
\end{align}
In view of \eqref{F}, \eqref{DI}, and the nonnegativity of $f_i,g_i$, and $h_i$,
\begin{align}
1_B\ &\E_x\left[F_i(\rho^+_R,\rho_T)\ \middle|\ \mathcal{F}_{\rho^+_S}\right]\notag \\
& \ge1_B \delta_i(\rho^+_S) \E_x\Big[\delta_i(\rho^+_R-\rho^+_S)f_i(X_{\rho^+_R})1_{\{\rho^+_R<\rho_T\}}+\delta_i(\rho^+_T-\rho^+_S)g_i(X_{\rho_T})1_{\{\rho^+_R>\rho_T\}}\notag \\
&\hspace{3in}+\delta_i(\rho^+_R-\rho^+_S)h_i(X_{\rho^+_R})1_{\{\rho^+_R=\rho_T\}} \Big|\ \mathcal{F}_{\rho^+_S}\Big]\notag\\
&= 1_B \delta_i(\rho^+_S) J_i\Big({X^x_{\rho^+_S}}, \rho^+_R, \rho_T\Big),  \label{smaller}
\end{align} 
where the equality follows from  the strong Markov property of $X$. Thanks to the above inequality and the fact that $1_B F_i(\rho_S^+,\rho_T) = 1_B \delta_i(\rho^+_S) f_i(X^x_{\rho^+_S})=1_B \delta_i(\rho^+_S) J_i(X^x_{\rho^+_S},0,\rho_T)$, \eqref{1_B} implies
\begin{align}\label{1_B'}
\E_x\left[1_B(F_i(\rho^+_R,\rho_T)-F_i(\rho^+_S,\rho_T))\right]&\geq\E_x\left[1_B\delta_i(\rho^+_S)\left(J_i(X_{\rho^+_S},\rho^+_R,\rho_T)-J_i(X_{\rho^+_S},0,\rho_T)\right)\right].
\end{align}
On the set $B$, we deduce from $\rho^+_S<\rho^+_R$ that $X^x_{\rho^+_S}\notin R$. Since $R\in\cE_i^T$,  \eqref{Theta} implies $J_i(X_{\rho^+_S},0,\rho_T)\le J_i(X_{\rho^+_S},\rho^+_R,\rho_T)$. Then, \eqref{1_B'} yields $\E_x\left[1_B(F_i(\rho^+_R,\rho_T)-F_i(\rho^+_S,\rho_T))\right]\ge 0$. On strength of this and \eqref{1_A}, we conclude from \eqref{1_A,B} that $J_i(x,\rho^+_R,\rho_T)-J_i(x,\rho^+_S,\rho_T)\ge0$.
\end{proof}

There is an intriguing message here---a {\it smaller} intra-personal equilibrium is more {\it rewarding}. Indeed, for any $R,S\in\cE^T_i$ with $R\subseteq S$, Lemma~\ref{l5} asserts $J_i(x,\rho^+_R,\rho_T)\geq J_i(x,\rho^+_S,\rho_T)$ for all $x\in\X$. This ``ranking by size'' will play a crucial role in Theorem~\ref{t1} below, where an {\it optimal} intra-personal equilibrium is derived. Economically, this ranking simply reflects a player's decreasing impatience, which is captured by \eqref{DI}; see the detailed discussion below \cite[Corollary 3.2]{HZ20}.


\section{The One-Player Analysis}\label{sec:one-player}
In this section, we first focus on developing an iterative procedure for each player that directly leads to her intra-personal equilibrium. Next, under a supermartingale condition, we will establish the monotonicity of this iterative procedure. 

For each $i\in\{1,2\}$, we introduce, for any fixed $T\in\B$, the operator $\Phi_i^T:\B\to\B$ defined by 
\begin{equation}\label{Phi_i^T}
\Phi_i^T(S):=S\cup\{x\notin S: J_i(x,0,\rho_T)>V_i^T(x,S)\},
\end{equation}
where
\begin{equation}\label{V_i^T}
V_i^T(x,S):=\sup_{1\leq\tau\leq\rho^+_S}\E_x[F_i(\tau,\rho_T)]\quad \hbox{for any}\ x\in\X\ \hbox{and}\ S\in\B.
\end{equation}
We first note that $V_i^T(x,\cdot)$ converges desirably along nondecreasing sequences of stopping policies.

\begin{lemma}\label{l2}
Fix $i\in\{1,2\}$. 
Let $(S_n)_{n\in\N}$ be a nondecreasing sequence in $\B$. Then, for any $T\in\B$, $V_i^T(x,S_n)\downarrow V_i^T(x,S_\infty)$ for all $x\in\X$, with $S_\infty:= \bigcup_{n\in\N} S_n$.
\end{lemma}

\begin{proof}
Fix $x\in\X$. By the definition of $V_i^T$ in \eqref{V_i^T}, $\lim_{n\to\infty}V_i^T(x,S_n)\geq V_i^T(x,S_\infty)$. To show the converse inequality, let $\tau_n\in\T$ with $1\le \tau_n\le \rho^+_{S_n}$ be a $\frac{1}{n}$-optimizer of $V_i^T(x,S_n)$, for each $n\in\N$. Then,
\begin{align}
V_i^T(x,S_n)-V_i^T(x,S_\infty)&\leq \E_x[F_i(\tau_n,\rho_T)]+{1}/{n}-\E_x[F_i(\tau_n\wedge\rho^+_{S_\infty},\rho_T)]\notag\\
&= \E_x\left[F_i(\tau_n,\rho_T)-F_i(\tau_n\wedge\rho^+_{S_\infty},\rho_T)\right]+1/n.\label{V-V} 
\end{align}
For each $\omega\in\Omega$, \lemref{l1} asserts the existence of $N\in\N$ such that $\tau_n(\omega)\leq\rho^+_{S_n}(\omega)=\rho^+_{S_\infty}(\omega)$ for all $n\ge N$. In particular, $\tau_n(\omega)= (\tau_n\wedge\rho^+_{S_\infty})(\omega)$ for all $n\ge N$. This, together with \eqref{dominated}, indicates that we may apply the dominated convergence theorem to show that the expectation in \eqref{V-V} converges to $0$. Hence, we conclude from \eqref{V-V} that  $\lim_{n\to\infty}V^T(x,S_n)-V^T(x,S_\infty)\le0$.
\end{proof}

For any fixed $T\in\B$, we perform an iterative procedure by starting with the empty set and applying the operator $\Phi_i^T$ repetitively. This will give an intra-personal equilibrium for Player $i$.

\begin{proposition}\label{l3}
Fix $i\in\{1,2\}$ and assume $h_i\leq g_i$. For any $T\in\B$, let $\left(S_i^n(T)\right)_{n\in\N}$ be a nondecreasing sequence in $\B$ defined by
\begin{equation}\label{S_n}
S_i^1(T) := \Phi_i^T(\emptyset)\quad \hbox{and}\quad S_i^n(T) := \Phi_i^T\left(S^{n-1}(T)\right)\ \ \hbox{for}\ n\ge 2,
\end{equation}
where $\Phi_i^T:\B\to\B$ is defined as in \eqref{Phi_i^T}. Then, 
\begin{equation}\label{S_infty}
\Gamma_i(T):=\bigcup_{n\in\N} S_i^n(T)\in\cE_i^T.
\end{equation} 
Moreover, $T\cap S_i^n(T)=\emptyset$ for all $n\in\N$; in particular, $T\cap\Gamma_i(T)=\emptyset$.  
\end{proposition}

\begin{proof}
Fix $x\in \Gamma_i(T)$. If $x\in S_i^1(T)= \Phi_i^T(\emptyset)$, by \eqref{Phi_i^T} we have $J_i(x,0,\rho_T)>V_i^T(x,\emptyset)\geq V_i^T(x,\Gamma_i(T))\geq J_i(x,\rho^+_{\Gamma_i(T)},\rho_T)$, which implies $x\in\Theta_i^T(\Gamma_i(T))$. If $x\notin S_i^1(T)$, since $(S_i^n(T))_{n\in\N}$ is by definition nondecreasing, there must exist $n\in\N$ such that $x\in S_i^{n+1}(T)\setminus S_i^n(T)$. Thanks again to \eqref{Phi_i^T}, this yields $J_i(x,0,\rho_T)>V_i^T(x,S_i^n(T))\geq V_i^T(x,\Gamma_i(T))\geq J_i(x,\rho^+_{\Gamma_i(T)},\rho_T)$, which implies $x\in\Theta_i^T(\Gamma_i(T))$. Hence, we conclude $\Gamma_i(T)\subseteq\Theta_i^T(\Gamma_i(T))$. 

It remains to show the converse inclusion. Fix $x\notin \Gamma_i(T)$. We claim that $x\notin \Theta_i^T(\Gamma_i(T))$, i.e. 
\begin{equation}\label{e1}
J_i(x,0,\rho_T)\leq J_i(x,\rho^+_{\Gamma_i(T)},\rho_T).
\end{equation}
Assume to the contrary that \eqref{e1} fails, so that 
\begin{equation}\label{Delta}
\Delta:=\{y\notin \Gamma_i(T): J_i(y,0,\rho_T)>J_i(y,\rho^+_{\Gamma_i(T)},\rho_T)\} \neq\emptyset. 
\end{equation}
Consider
\begin{equation}\label{alpha}
\alpha:=\sup_{y\in\Delta}\left\{J_i(y,0,\rho_T)-J_i(y,\rho^+_{\Gamma_i(T)},\rho_T)\right\}>0.
\end{equation}
As $\delta_i(1)<\delta_i(0)=1$, we can take $x^*\in\Delta$ such that
\begin{equation}\label{e2}
J_i(x^*,0,\rho_T)-J_i(x^*,\rho^+_{\Gamma_i(T)},\rho_T)>\frac{1+\delta_i(1)}{2}\alpha.
\end{equation}
Observe that we must have $x^*\notin T$, and thus $\rho_T>0$ $\P_{x^*}$-a.s. Indeed, if $x^*\in T$, then $J_i(x^*,0,\rho_T)-J_i(x^*,\rho^+_{\Gamma_i(T)},\rho_T)=h_i(x)-g_i(x)\leq 0$, which contradicts \eqref{e2}. Moreover, since $x^*\notin \Gamma_i(T)$ implies $x^*\notin S_i^{n}(T)$ for all $n\in\N$, we deduce from \eqref{Phi_i^T} that $J_i(x^*,0,\rho_T)\leq V_i^T(x^*,S_i^n(T))$ for all $n\in\N$. By \lemref{l2}, this implies 
\begin{equation}\label{J<V^T}
J_i(x^*,0,\rho_T)\leq V_i^T(x^*,\Gamma_i(T)).
\end{equation}
Let $\rho^*\in\T$ with $1\le \rho^*\le \rho^+_{\Gamma_i(T)}$ be a $\frac{1-\delta_i(1)}{2}\alpha$-optimizer of $V_i^T(x^*,\Gamma_i(T))$. Consider the sets
\begin{align*}
A&:=\{\omega\in\Omega: \rho^*(\omega)<\rho^+_{\Gamma_i(T)}(\omega),\ \rho^*(\omega)\leq\rho_T(\omega),\ X_{\rho^*}(\omega)\notin(\Gamma_i(T)\cup\Delta)\},\\
B&:=\{\omega\in\Omega: \rho^*(\omega)<\rho^+_{\Gamma_i(T)}(\omega),\ \rho^*(\omega)\leq\rho_T(\omega),\ X_{\rho^*}(\omega)\in\Delta\setminus \Gamma_i(T)\}.
\end{align*}
By \eqref{e2} and \eqref{J<V^T}, 
\begin{align}
\frac{1+\delta_i(1)}{2}\alpha&<J_i(x^*,0,\rho_T)-J_i(x^*,\rho^+_{\Gamma_i(T)},\rho_T)\notag\\
&\leq\E_{x^*}\left[F_i(\rho^*,\rho_T)-F_i(\rho^+_{\Gamma_i(T)},\rho_T)\right]+\frac{1-\delta_i(1)}{2}\alpha\notag\\
&=\E_{x^*}\left[(1_A+1_B)(F_i(\rho^*,\rho_T)-F_i(\rho^+_{\Gamma_i(T)},\rho_T))\right]+\frac{1-\delta_i(1)}{2}\alpha.\label{11}
\end{align}
By \eqref{DI} and the nonnegativity of $f_i$, $g_i$, and $h_i$, we can argue as in \eqref{smaller} to get
\begin{align*}
\E_{x^*}\left[(1_A+1_B) F_i(\rho^+_{\Gamma_i(T)},\rho_T)\right] &= \E_{x^*}\left[(1_A+1_B)\E_{x^*}\left[ F_i(\rho^+_{\Gamma_i(T)},\rho_T)\ \middle|\ \F_{\rho^*}\right]\right]\\
&\ge \E_{x^*}\left[(1_A+1_B) \delta_i(\rho^*) J_i(X_{\rho^*},\rho^+_{\Gamma_i(T)},\rho_T)\right].   
\end{align*}
Thanks to this and the fact that $(1_{A}+1_B) F_i(\rho^*,\rho_T) = (1_{A}+1_B) \delta_i(\rho^*) J_i(X_{\rho^*}, 0,\rho_T)$, \eqref{11} yields
\begin{align*}
\frac{1+\delta_i(1)}{2}\alpha&<\E_{x^*}\left[(1_A+1_B)\delta_i(\rho^*)(J_i(X_{\rho^*},0,\rho_T)-J_i(X_{\rho^*},\rho^+_{\Gamma_i(T)},\rho_T))\right]+\frac{1-\delta_i(1)}{2}\alpha\\
&\leq\E_{x^*}\left[1_B\delta_i(\rho^*)(J_i(X_{\rho^*},0,\rho_T)-J_i(X_{\rho^*},\rho^+_{\Gamma_i(T)},\rho_T))\right]+\frac{1-\delta_i(1)}{2}\alpha\\
&\leq\delta_i(1)\alpha+\frac{1-\delta_i(1)}{2}\alpha=\frac{1+\delta_i(1)}{2}\alpha. 
\end{align*}
where the second inequality follows from $X_{\rho^*}\notin \Delta$ on $A$ and the definition of $\Delta$ in \eqref{Delta}, and the third inequality is due to $X_{\rho^*}\in \Delta$ on $B$, the definition of $\alpha$ in \eqref{alpha}, and $\rho^*\ge 1$ by definition.  The above inequality is clearly a contradiction, and we thus conclude that \eqref{e1} holds. That is, we have shown that $(\Gamma_i(T))^c \subseteq (\Theta(\Gamma_i(T)))^c$, or simply $\Gamma_i(T) \supseteq \Theta(\Gamma_i(T))$. This brings the final conclusion that $\Gamma_i(T)=\Theta_i^T(\Gamma_i(T))$, i.e. $\Gamma_i(T)\in\cE_i^T$.

Finally, for any $x\in T$, since $\rho_T=0$ $\P_x$-a.s., 
\begin{equation}\label{not in S^n}
J_i(x,0,\rho_{T}) = h_i(x) \le g_i(x)=\sup_{1\le \tau\le \rho^+_{R}} \E_x[F_i(\tau,\rho_T)]  = V_i^{T}(x,R)\quad \forall R\in\B.
\end{equation}
In view of \eqref{S_n}, taking $R=\emptyset$ in \eqref{not in S^n} shows that $x\notin S_i^1(T)$. By the fact $x\notin S_i^1(T)$ and taking $R=S_i^1(T)$ in \eqref{not in S^n}, we in turn obtain $x\notin S_i^2(T)$. Applying \eqref{not in S^n} recursively in the same way then gives $x\notin S_i^n(T)$ for all $n\in\N$. Hence, we conclude $T\cap S_i^n(T)$ for all $n\in\N$. 
\end{proof}

\begin{remark}
The idea of the iterative procedure \eqref{S_n} was initially inspired by \cite[Section 3]{CL18}, while its specific construction is partially borrowed from \cite[Theorem 2.2]{https://doi.org/10.1111/mafi.12293}.
\end{remark}

We will go one step further to claim that $\Gamma_i(T)$ in \eqref{S_infty} is in fact an {\it optimal} intra-personal equilibrium for Player $i$. To this end, we need the following auxiliary result: Being included in an intra-personal equilibrium is an invariant relation under the operator $\Phi^T_i$.

\begin{lemma}\label{l4}
Fix $i\in\{1,2\}$, $T\in\B$, and $R\in\cE_i^T$. For any $S\in \B$ with $S\subseteq R$, $\Phi^T_i(S)\subseteq R$.
\end{lemma}

\begin{proof}
By contradiction, suppose that there exists $x\in\Phi_i^T(S)\setminus R$ for some $S\in \B$ with $S\subseteq R$. With $x\in\Phi_i^T(S)$ but $x\notin S$, \eqref{Phi_i^T} gives
$
J_i(x,0,\rho_T)> V_i^T(x,S)\geq V_i^T(x,R)\geq J_i(x,\rho^+_R,\rho_T),
$
where the second and third inequalities follow directly from the definition of $V_i^T$ in \eqref{V_i^T}. As $x\notin R$, this shows that $x\in \Theta_i^T(R)$, from which we conclude $\Theta_i^T(R)\neq R$. This contradicts $R\in\cE_i^T$. 
\end{proof}

Now, we are ready to present the main result of this section.

\begin{theorem}\label{t1}
Fix $i\in\{1,2\}$ and assume $h_i\leq g_i$. For any $T\in\B$, $\Gamma_i(T)$ defined in \eqref{S_infty} belongs to $\widehat\cE_i^T$.
\end{theorem}

\begin{proof}
By \propref{l3}, $\Gamma_i(T)\in \cE_i^T$. Recall $(S_i^n(T))_{n\in\N}$ defined in \eqref{S_n}. For any $R\in\cE_i^T$, \lemref{l4} directly implies $S_i^1(T)=\Phi_i^T(\emptyset) \subseteq R$. By applying \lemref{l4}  recursively, we get $S_i^n(T)=\Phi_i^T(S_i^{n-1}(T)) \subseteq R$ for all $n\ge 2$. Hence, $\Gamma_i(T)=\bigcup_{n\in\N} S_i^n(T)\subseteq R$. This, together with  \lemref{l5}, gives $J_i(x,\rho^+_{\Gamma_i(T)},\rho_T)\geq J_i(x,\rho^+_R,\rho_T)$, and thus $U_i^T(x,\Gamma_i(T))\ge U_i^T(x,R)$, for all $x\in\X$. As $R\in\cE_i^T$ is arbitrarily chosen, we conclude $\Gamma_i(T)\in \widehat \cE_i^T$.  
\end{proof}


\subsection{Monotonicity with respect to $T\in\B$}
So far, we have fixed $T\in\B$ (the other player's stopping policy) and constructed a corresponding optimal intra-personal equilibrium $\Gamma_i(T)$ in \eqref{S_infty}. By viewing $T\in\B$ as a variable, we will show that the map $T\mapsto\Gamma_i(T)$ is monotone under appropriate conditions. 
 
\begin{lemma}\label{l6}
Fix $i\in\{1,2\}$. Assume $f_i\leq h_i\leq g_i$ and that 
\begin{equation}\label{supermartingale}
(\delta_i(t)g_i(X_t^x))_{t\geq 0}\ \  \hbox{is a supermartingale for all $x\in\X$}. 
\end{equation}
Then, for any $T, R\in\B$ with $T\subseteq R$, 
\begin{equation}\label{J<J}
J_i(x,\tau,\rho_T)\leq J_i(x,\tau,\rho_R)\quad\forall x\in\X\ \hbox{and}\ \tau\in\mathcal{T}.
\end{equation}
Hence, $V_i^T(x,S)\leq V_i^R(x,S)$ for all $x\in\mathbb{X}$ and $S\in\B$.
Moreover, we have
\begin{equation}\label{Phi inclusion}
\Phi_i^T(S)\supseteq \Phi_i^R(S')\quad \forall S, S'\in\B\ \hbox{with}\ S\supseteq S'. 
\end{equation}
\end{lemma}

\begin{proof}
Given $x\in\X$ and $\tau\in\mathcal{T}$, consider $A:=\{\omega\in\Omega:\tau=\rho_R<\rho_T\}$ and $B:=\{\omega\in\Omega:\rho_R<\tau\wedge\rho_T\}.$
Observe that
\begin{align*}
\E_x\left[1_B(F_i(\tau,\rho_T)-F_i(\tau,\rho_R))\right]&=\E_x\left[1_B(F_i(\tau,\rho_T)-\delta_i(\rho_R)g_i(X_{\rho_R}))\right]\\
&\leq\E_x\left[1_B(\delta_i(\tau\wedge\rho_T)g_i(X_{\tau\wedge\rho_T})-\delta_i(\rho_R)g_i(X_{\rho_R}))\right]\\
&=\E_x\left[1_B\left(\E\left[\delta_i(\tau\wedge\rho_T)g_i(X_{\tau\wedge\rho_T})\Big|\mathcal{F}_{\rho_R}\right]-\delta_i(\rho_R)g_i(X_{\rho_R})\right)\right]\leq 0,
\end{align*}
where the first inequality is due to $f_i\leq h_i\leq g_i$ and the last inequality follows from $(\delta_i(t)g_i(X_t))_{t\geq 0}$ being a supermartingale. By the above inequality and 
\[
\E_x\left[1_A(F_i(\tau,\rho_T)-F_i(\tau,\rho_R))\right]=\E_x\left[1_A(\delta_i(\tau)f_i(X_\tau)-\delta_i(\tau)h_i(X_\tau))\right]\leq 0,
\]
thanks to $f_i\le h_i$, we conclude $J_i(x,\tau,\rho_T)-J_i(x,\tau,\rho_R)=\E_x\left[(1_A+1_B)(F_i(\tau,\rho_T)-F_i(\tau,\rho_R))\right]\le 0$.

Next, fix $S, S'\in\B$ with $S\supseteq S'$. For any $x\in \Phi_i^R(S')$, if $x\in S$, then $x\in \Phi_i^T(S)$ by definition. Hence, we assume $x\notin S$ in the following. With $x\in \Phi^R(S')\setminus S'$, \eqref{Phi_i^T} gives
\begin{equation}\label{taka}
J_i(x,0,\rho_R)>V_i^R(x,S')\geq V_i^R(x,S)\geq V_i^T(x,S),
\end{equation}
where the last inequality follows from \eqref{J<J} and \eqref{V_i^T}. Note that we must have $x\notin R$. Indeed, if $x\in R$, then $\rho_R=0$ $\P_x$-a.s. and thus 
\[
J_i(x,0,\rho_R)=h_i(x)\le g_i(x) =\sup_{1\le \tau\le \rho^+_{S'}} \E_x[F_i(\tau,\rho_R)] =V_i^R(x,S'),
\]
which contradicts the first inequality in \eqref{taka}. With $x\notin R$ and thus $\rho_T\ge \rho_R>0$ $\P_x$-a.s., 
$$J_i(x,0,\rho_T) = f_i(x)= J_i(x,0,\rho_R)>V_i^T(x,S),$$
where the last inequality follows from \eqref{taka}. This, together with $x\notin S$, yields $x\in \Phi_i^T(S)$. We therefore conclude $\Phi_i^R(S')\subseteq\Phi_i^T(S)$.
\end{proof}


\begin{corollary}\label{l8}
Fix $i\in\{1,2\}$. Assume $f_i\leq h_i\leq g_i$ and \eqref{supermartingale}. 
Then, for any $T, R\in\B$ with $T\subseteq R$, we have $\Gamma_i(T)\supseteq \Gamma_i(R)$, with $\Gamma_i(\cdot)$ defined as in \eqref{S_infty}.
\end{corollary}

\begin{proof}
In view of \eqref{S_infty}, $\Gamma_i(T)=\bigcup_{n\in\N} S_i^n(T)$ and $\Gamma_i(R)=\bigcup_{n\in\N} S_i^n(R)$, with $(S_i^n(T))_{n\in\N}$ and $(S_i^n(R))_{n\in\N}$ defined as in \eqref{S_n}. Hence, it suffices to show $S_i^n(T)\supseteq S_i^n(R)$ for all $n\in\N$. By \eqref{Phi inclusion}, $S^1_i(T) = \Phi_i^T(\emptyset)\supseteq\Phi_i^R(\emptyset) = S^1_i(R)$. Using this and \eqref{Phi inclusion} again, we get $S^2_i(T) = \Phi_i^T(S_i^1(T))\supseteq\Phi_i^R(S^1_i(R)) = S_i^2(R)$. Applying \eqref{Phi inclusion}  recursively in the same way then yields $S^n_i(T)\supseteq S^n_i(R)$ for all $n\in\N$.
\end{proof}

The monotonicity of $T\mapsto \Gamma_i(T)$, $i\in\{1,2\}$, will play a crucial role in Theorem~\ref{t2} below, contributing to the convergence of an alternating iterative procedure performed jointly by Players 1 and 2. 


\section{The Existence of Inter-Personal Equilibria}\label{sec:existence}
In this section, we will design an alternating iterative procedure, to be performed jointly by the two players. As shown in Theorem~\ref{t2} below, this procedure converges to a soft inter-personal equilibrium that is almost sharp. By a probabilistic modification of this iterative procedure and an appropriate use of Zorn's lemma, we establish the existence of sharp inter-personal equilibria in Theorem~\ref{t3} below. Explicit examples will be presented to illustrate this iterative procedure and the necessity of a supermartingale condition. 

First, we observe that Theorem~\ref{t1} already provides a sufficient condition for the existence of a sharp inter-personal equilibrium. 

\begin{lemma}\label{lem:sufficient}
For each $i\in\{1,2\}$, assume $h_i\le g_i$. If $(S,T)\in \B\times\B$ satisfies 
\begin{equation}\label{suffi. cond.}
\Gamma_1(T)=S\quad \hbox{and}\quad \Gamma_2(S)=T, 
\end{equation}
then $(S,T)\in \widehat\cE$. 
\end{lemma}

\begin{proof}
By Theorem~\ref{t1}, $S=\Gamma_1(T)\in\widehat\cE_1^T$ and $T=\Gamma_2(S)\in\widehat\cE_2^S$, so that $(S,T)\in\widehat\cE$. 
\end{proof}


\subsection{Construction of Soft Inter-Personal Equilibira}

In order to achieve \eqref{suffi. cond.}, we let the two players take turns to perform the {\it individual} iterative procedure \eqref{S_n}. As the next result shows, such alternating iterations do converge, and the limit is guaranteed a soft inter-personal equilibrium.

\begin{theorem}\label{t2}
For each $i\in\{1,2\}$, assume $f_i\le h_i\le g_i$ and \eqref{supermartingale}. 
Let $(S_n,T_n)$ be a sequence in $\B\times\B$ defined by $S_0:=\emptyset$ and 
\begin{equation}\label{alternating}
T_n:=\Gamma_2(S_n)\qq S_{n+1}:=\Gamma_1(T_n),\quad \forall n\in\N\cup\{0\}. 
\end{equation}
Then, $(S_n)$ is nondecreasing and $(T_n)$ is nonincreasing. By taking $S_\infty:=\bigcup_n S_n$ and $T_\infty:=\bigcap_n T_n$, we have $(S_\infty,T_\infty)\in \cE$ with 
\begin{equation}\label{almost sharp}
\Gamma_1(T_\infty)=S_\infty\quad \hbox{and}\quad \Gamma_2(S_\infty)\subseteq T_\infty.
\end{equation}
\end{theorem}

\begin{proof}
As $S_0=\emptyset\subseteq S_1$, applying Corollary~\ref{l8} for Player 2 implies $T_0=\Gamma_2(S_0)\supseteq \Gamma_2(S_1)= T_1$. With $T_0\supseteq T_1$, applying Corollary~\ref{l8} for Player 1 implies $S_1=\Gamma_1(T_0)\subseteq \Gamma_1(T_1)= S_2$. Again, by $S_1\subseteq S_2$, applying Corollary~\ref{l8} for Player 2 gives $T_1=\Gamma_2(S_1)\supseteq \Gamma_2(S_2)= T_2$. Repeating this procedure for Players 1 and 2 recursively, we have $(S_n)$ nondecreasing and $(T_n)$ is nonincreasing.

Next, let us show that $\Theta_1^{T_\infty}(S_\infty)=S_\infty$. 
Fix $x\in S_\infty= \bigcup_n S_n$. There exists $N\in\N$ such that $x\in S_{n+1}=\Gamma_1(T_n)$ for all $n>N$. By the fact that $\Gamma_1(T_n) \in \cE_1^{T_n}$ (thanks to Proposition~\ref{l3}), $J_1(x,0,\rho_{T_n})\geq J_1(x,\rho^+_{S_{n+1}},\rho_{T_n})$ for all $n\ge N$. As $n\to\infty$, \eqref{dominated} allows us to use the dominated convergence theorem, so that we may conclude from \lemref{l1} that
\begin{equation}\label{use dominated'}
J_1(x,0,\rho_{T_\infty})\geq J_1(x,\rho^+_{S_\infty},\rho_{T_\infty}),
\end{equation}
which implies $x\in \Theta_1^{T_\infty}(S_\infty)$. Hence, $S_\infty\subseteq \Theta_1^{T_\infty}(S_\infty)$. On the other hand, for any $x\notin S_\infty = \bigcup_n S_n$, $x\notin S_{n+1}=\Gamma_1(T_n)$ for all $n\in\N$. Thanks again to the fact that $\Gamma_1(T_n) \in \cE_1^{T_n}$, $x\notin \Gamma_1(T_n)$ indicates
$J_1(x,0,\rho_{T_n})\leq J_1(x,\rho^+_{S_{n+1}},\rho_{T_n})$ for all $n\in\N$. As $n\to\infty$, we can argue as in \eqref{use dominated'} to get $J_1(x,0,\rho_{T_\infty})\leq J_1(x,\rho^+_{S_\infty},T_\infty)$, which implies $x\notin \Theta_1^{T_\infty}(S_\infty)$. Hence, $(S_\infty)^c\subseteq \big(\Theta_1^{T_\infty}(S_\infty)\big)^c$, or $\Theta_1^{T_\infty}(S_\infty)\subseteq S_\infty$. We thus conclude $\Theta_1^{T_\infty}(S_\infty)=S_\infty$, i.e. $S_\infty\in \cE_1^{T_\infty}$. Similar arguments as above yield $\Theta_2^{S_\infty}(T_\infty)=T_\infty$, i.e. $T_\infty\in \cE_2^{S_\infty}$. This readily shows that $(S_\infty,T_\infty)\in\cE$. 

As $S_n\subseteq S_\infty$ by construction for all $n\in\N$, Corollary~\ref{l8} implies $\Gamma_2(S_\infty)\subseteq \Gamma_2(S_n)=T_n$ for all $n\in\N$, which in turn gives $\Gamma_2(S_\infty)\subseteq \bigcap_n T_n=T_\infty$. Similarly, as $T_n \supseteq T_\infty$ by construction for all $n\in\N$, Corollary~\ref{l8} implies $\Gamma_1(T_\infty)\supseteq \Gamma_1(T_n)=S_{n+1}$ for all $n\in\N$, which in turn gives $\Gamma_1(T_\infty)\supseteq \bigcup_{n\in\N} S_{n+1}=S_\infty$. Finally, recall $S_\infty\in \cE_1^{T_\infty}$. This, together with Lemma~\ref{l4}, implies $\Gamma_1(T_\infty)\subseteq S_\infty$. We then conclude $\Gamma_1(T_\infty)= S_\infty$. 
\end{proof} 

The next example illustrates the alternating iterative procedure \eqref{alternating} explicitly.

\begin{example}\label{eg1}
Let $\X$ contain countably many states, i.e. $\mathbb{X}=\{x_0,x_1, x_2, \dotso\}$, and assume
\begin{equation}\label{transition}
\begin{split}
&\P_{x_{n+1}}(X_1=x_n)=1,\quad \hbox{for}\ n=0,1,2\dotso,\\
\P_{x_0}(X_1=x_0)&=1-\eps\quad\text{and}\quad\P_{x_0}(X_1=x_1)=\eps,\quad \hbox{for some}\ \eps\in[0,1).
\end{split}
\end{equation}
Take $M>1$ such that 
\begin{equation}\label{delta_2}
\delta_2(2)<1/M< \delta_2(1).
\end{equation}
Additionally, take $L>1$ and consider the following payoff functions for the two players
\begin{align}
f_1(x_n)= 1\quad\text{and}\quad g_1(x_n)= L\quad \hbox{for}\ n=0,1,2,\dotso,\label{p1} \\
f_2(x_0)=0,\quad f_2(x_n)=1\ \hbox{for}\ n=1,2,\dotso,\quad\text{and}\quad g_2(x_n)=M\ \hbox{for}\ n=0,1,2\dotso,\label{p2}
\end{align}
while $h_1$ (resp. $h_2$) is allowed to be any function such that $f_1\le h_1\le g_1$ (resp. $f_2\le h_2\le g_2$) on $\X$. 

For $\eps\in[0,1)$ small enough, we claim that the alternating iterative procedure \eqref{alternating} gives rise to 
\begin{align}
S_0&=\emptyset, &T_0&=\{x_1,x_2,\dotso\},\notag\\
S_1&=\{x_0\}, &T_1&=\{x_2,x_3,\dotso\},\notag\\
S_2&=\{x_0,x_1\}, &T_2&=\{x_3,x_4,\dotso\},\label{eg1 S_n T_n}\\
&\vdots& &\vdots\notag\\
S_n&=\{x_0,x_1,\dotso,x_{n-1}\},&T_n&=\{x_{n+1},x_{n+2},\dotso\}.\notag
\end{align}
First, starting with $S_0=\emptyset$, we deduce from \eqref{p2} that for any $S\in \B=2^\X$, 
\begin{equation}\label{1in}
V^{S_0}_2(x_n,S) = \sup_{1\le \tau\le \rho^+_S} \E_{x_n}[F_2(\tau,\rho_{S_0})]
\begin{cases}
\ge 0 =f_2(x_n)=J_2(x_n,0,\rho_{S_0}),\  &\hbox{for}\ n=0,\\
<1 =f_2(x_n)=J_2(x_n,0,\rho_{S_0}),\  &\hbox{for}\ n=1,2,\dotso
\end{cases}
\end{equation}
This implies $\Phi^{S_0}_2(\emptyset) = \{x_1,x_2,...\}$ and $\Phi^{S_0}_2(\{x_1,x_2,...\}) = \{x_1,x_2,...\}$, so that $T_0 := \Gamma_2(S_0)=\{x_1,x_2,...\}$. Next, thanks to \eqref{transition} and \eqref{p1}, for any $S\in 2^\X$, 
\begin{align}
V^{T_0}_1(x_n,S)& = \sup_{1\le \tau\le \rho^+_S} \E_{x_n}[F_1(\tau,\rho_{T_0})]\notag\\
&\hspace{-0.5in}\begin{cases}
\le (1-\eps)\delta_1(1)+\eps L \left(1+(1-\eps)+(1-\eps)^2+\dotso\right)\\
\hspace{1in}= (1-\eps)\delta_1(1)+\frac{\eps L}{1+\eps} <1 =f_1(x_n)= J_1(x_n,0,\rho_{T_0}),\ &\hbox{for}\ n=0,\label{2in}\\
=g_1(x_n)\ge h_1(x_n) = J_1(x_n,0,\rho_{T_0}),\ &\hbox{for}\ n=1,2,\dotso
\end{cases}
\end{align}
where the inequality ``$(1-\eps)\delta_1(1)+\frac{\eps L}{1+\eps} <1$"  holds as $\eps\in [0,1)$ is small enough. This implies $\Phi^{T_0}_1(\emptyset) = \{x_0\}$ and $\Phi^{T_0}_1(\{x_0\}) = \{x_0\}$, so that $S_1 := \Gamma_1(T_0)=\{x_0\}$. Now, thanks to \eqref{transition} and \eqref{p2}, for any $S\in 2^\X$ such that $x_0\notin S$, 
\begin{align}
V^{S_1}_2(x_n,S)& = \sup_{1\le \tau\le \rho^+_S} \E_{x_n}[F_2(\tau,\rho_{S_1})]\notag\\
&\hspace{-0.5in}\begin{cases}
=g_2(x_n)\ge h_2(x_n) = J_2(x_n,0,\rho_{S_1}),\ &\hbox{for}\ n=0,\\
= \delta_2(1) g_2(x_0) = \delta_2(1) M \ge1=f_2(x_n) = J_2(x_n,0,\rho_{S_1}),\ &\hbox{for}\ n=1,\\
\le \max\{\delta_2(1), \delta_2(2) M\}<1 = f_2(x_n) = J_2(x_n,0,\rho_{S_1}),\ &\hbox{for}\ n=2,3, \dotso\label{3in} 
\end{cases}
\end{align}
where we use \eqref{delta_2} in the last two lines.
This implies $\Phi^{S_1}_2(\emptyset) = \{x_2, x_3, ...\}$ and $\Phi^{S_1}_2(\{x_2, x_3,...\}) = \{x_2,x_3,...\}$, so that $T_1 := \Gamma_2(S_1)=\{x_2,x_3,...\}$. By similar arguments as above, we can derive $S_n$ and $T_n$ in \eqref{eg1 S_n T_n} for all $n\ge 2$. 
By Theorem~\ref{t2}, $(S_\infty,T_\infty):= (\bigcup_{n\in\N} S_n, \bigcap_{n\in\N} T_n) = (\X,\emptyset)$ is a soft inter-personal equilibrium and satisfies \eqref{almost sharp}.
Observe that for any $n=0,1,2,\dotso$, 
\[
V^{\X}_2(x_n,\emptyset) = \sup_{1\le \tau\le \rho^+_\emptyset} \E_{x_n}[F_2(\tau,\rho_{\X})]=g_2(x_n)\ge h_2(x_n) = J_2(x_n,0,\rho_\X).
\]
It follows that $\Phi^\X_2(\emptyset)=\emptyset$, so that $\Gamma_2(S_\infty) =\Gamma_2(\X)=\emptyset=T_\infty$. That is, we have a stronger version of \eqref{almost sharp} where the inclusion therein is equality. Hence, by Lemma~\ref{lem:sufficient}, $(S_\infty,T_\infty)=(\X,\emptyset)$ is in fact a sharp inter-personal equilibrium.  
\end{example}

In view of \eqref{almost sharp} and Lemma~\ref{lem:sufficient}, the soft inter-personal equilibrium $(S_\infty,T_\infty)$ constructed in Theorem~\ref{t2} is {\it nearly} a sharp one. It is natural to ask whether the inclusion in \eqref{almost sharp} is actually equality (as in Example~\ref{eg1}), so that $(S_\infty,T_\infty)$ is sharp in general. The next example shows that this is generally {\it not} the case: the inclusion in \eqref{almost sharp} can be strict and $(S_\infty,T_\infty)$ may fail to be sharp.

\begin{example}\label{eg2}
Let us extend the state space in Example~\ref{eg1} by including two additional states, i.e. $\mathbb{X}=\{x_0,x_1,x_2,\dotso\}\cup\{y,z\}$. The transition probabilities are specified as in \eqref{transition}, as well as 
\begin{equation}\label{transition'}
\P_y(X_1=x_n)=p_n>0\ \ \text{with}\ \ \sum_{n=0}^\infty p_n=1\quad\text{and}\quad\P_z(X_1=y)=1.
\end{equation}
Take $M>1$ such that \eqref{delta_2} holds. Assume additionally that $\delta_2:[0,\infty)\to[0,1]$ satisfies 
$\delta_2(1)^2<\delta_2(2)$. 
Take $L>1$ and define $f_i$ and $g_i$, $i\in\{1,2\}$, as in \eqref{p1}-\eqref{p2} on $\{x_0,x_1,x_2,\dotso\}$, along with
\begin{equation}\label{p2'}
f_2(y)=M\delta_2(1),\quad f_2(z) \in \left(M\delta_2(1)^2\vee \delta_2(2), M\delta_2(2)\right),
\quad g_2(y)=g_2(z)=M,
\end{equation}
while $f_1$ and $g_1$ are allowed to take arbitrary nonnegative values on $\{y,z\}$ as long as $f_1\le g_1$. Also, $h_1$ (resp. $h_2$) is allowed to be any function such that $f_1\le h_1\le g_1$ (resp. $f_2\le h_2\le g_2$) on $\X$.

For $\eps\in [0,1)$ small enough, we claim that the alternating iterative procedure \eqref{alternating} gives rise to 
\begin{align}
S_0&=\emptyset,& T_0&=\{x_1,x_2,\dotso\}\cup\{y,z\},\notag\\
S_1&=\{x_0\},& T_1&=\{x_2,x_3,\dotso\}\cup\{y,z\},\notag\\
S_2&=\{x_0,x_1\},& T_2&=\{x_3,x_4,\dotso\}\cup\{y,z\},\label{eg2 S_n T_n}\\
&\vdots& &\vdots\notag\\
S_n&=\{x_0,x_1,\dotso,x_{n-1}\},&T_n&=\{x_{n+1},x_{n+2},\dotso\}\cup\{y,z\}.\notag
\end{align}
First, note that the relations \eqref{1in}, \eqref{2in}, and \eqref{3in} remain true in our current setting. Now, starting with $S_0=\emptyset$, we deduce from \eqref{transition'}, \eqref{p2}, and \eqref{p2'} that for any $S\in \B=2^\X$, 
\begin{align*}
V^{S_0}_2(y,S) &= \sup_{1\le \tau\le \rho^+_S} \E_{y}[F_2(\tau,\rho_{S_0})] <\delta_2(1)<f_2(y)=J_2(y,0,\rho_{S_0}),\\
V^{S_0}_2(z,S) &= \sup_{1\le \tau\le \rho^+_S} \E_{z}[F_2(\tau,\rho_{S_0})] \le \max\{\delta_2(1) f_2(y), \delta_2(2)\}\\
&\hspace{1.55in} =\max\{M \delta_2(1)^2, \delta_2(2)\}< f_2(z)=J_2(z,0,\rho_{S_0}).
\end{align*}
These two inequalities, along with \eqref{1in}, imply $\Phi^{S_0}_2(\emptyset) = \{x_1,x_2,...\}\cup\{y,z\}$ and $\Phi^{S_0}_2(\{x_1,x_2,...\}\cup\{y,z\}) = \{x_1,x_2,...\}\cup\{y,z\}$, so that $T_0 := \Gamma_2(S_0)=\{x_1,x_2,...\}\cup\{y,z\}$. Next, since $\{y,z\}\subset T_0$, for any $S\in 2^\X$, $V^{T_0}_1(x,S) = \sup_{1\le \tau\le \rho^+_S} \E_{x}[F_1(\tau,\rho_{T_0})] = g_1(x)\ge h_1(x) = J_1(x,0,\rho_{T_0})$ for $x\in\{y,z\}$. 
This, together with \eqref{2in}, implies that as $\eps\in[0,1)$ is small enough, $\Phi^{T_0}_1(\emptyset) = \{x_0\}$ and $\Phi^{T_0}_1(\{x_0\}) = \{x_0\}$, so that $S_1 := \Gamma_1(T_0)=\{x_0\}$.
Thanks to \eqref{transition'}, \eqref{transition}, \eqref{p2}, and \eqref{p2'}, for any $S\in 2^\X$ such that $x_0\notin S$, 
\begin{align*}
V^{S_1}_2(y,S)& = \sup_{1\le \tau\le \rho^+_S} \E_{y}[F_2(\tau,\rho_{S_1})]<\delta_2(1) M=f_2(y) = J_2(y,0,\rho_{S_1}),\\
V^{S_1}_2(z,S)& = \sup_{1\le \tau\le \rho^+_S} \E_{z}[F_2(\tau,\rho_{S_1})]\le \max\left\{\delta_2(1) f_2(y), \delta_2(2) \big((1-O(\eps))+O(\eps)M\big) \right\}\\
&\hspace{1.55in} <f_2(z) = J_2(z,0,\rho_{S_1}), 
\end{align*}
where the last inequality holds as $\eps\in[0,1)$ is small enough, thanks to \eqref{p2'}. The above two inequalities, along with \eqref{3in}, imply $\Phi^{S_1}_2(\emptyset) = \{x_2, x_3, ...\}\cup\{y,z\}$ and $\Phi^{S_1}_2(\{x_2, x_3,...\}) = \{x_2,x_3,...\}\cup\{y,z\}$, so that $T_1 := \Gamma_2(S_1)=\{x_2,x_3,...\}\cup\{y,z\}$. By similar arguments as above, we can derive $S_n$ and $T_n$ in \eqref{eg2 S_n T_n} for all $n\ge 2$. 
Hence, $(S_\infty,T_\infty):= (\bigcup_{n\in\N} S_n, \bigcap_{n\in\N} T_n) = (\{x_0,x_1,x_2,\dotso\},\{y,z\})$. 

Now, it can be easily checked that $V^{S_\infty}_2(x_n,\emptyset) = g_2(x_n)\ge h_2(x_n) = J_2(x_n,0,\rho_{S_\infty})$ for all $n=0,1,2,\dotso$. Moreover, due to \eqref{transition'},
\begin{align*}
V^{S_\infty}_2(y,\emptyset)& = \sup_{1\le \tau\le \rho^+_\emptyset} \E_{y}[F_2(\tau,\rho_{S_\infty})]=\delta_2(1) M=f_2(y) = J_2(y,0,\rho_{S_\infty}),\\
V^{S_\infty}_2(z,\emptyset)& = \sup_{1\le \tau\le \rho^+_\emptyset} \E_{z}[F_2(\tau,\rho_{S_\infty})]=\delta_2(2)M>f_2(z) = J_2(z,0,\rho_{S_\infty}).
\end{align*}
In view of \eqref{Phi_i^T}, we conclude $\Phi_2^{S_\infty}(\emptyset) = \emptyset$. By \eqref{S_infty}, this in turn implies  
$\Gamma_2(S_\infty)=\emptyset\subsetneq T_\infty.$
That is, the inclusion in \eqref{almost sharp} is strict, so that we can no longer conclude from Lemma~\ref{lem:sufficient} that $(S_\infty, T_\infty)$ is a sharp inter-personal equilibrium.  
In fact, $(S_\infty, T_\infty)$ is not sharp. Recall from Theorem~\ref{t1} that $\emptyset=\Gamma_2(S_\infty)\in \widehat \cE^{S_\infty}_2$. Then, it can be checked directly that $T_\infty=\{y,z\}\in \cE^{S_\infty}_2$ but $T_\infty\notin \widehat \cE^{S_\infty}_2$. Specifically, $T_\infty$ is strictly dominated by $\emptyset$ at the state $z$, as \eqref{p2'} indicates  
\begin{align*}
U^{S_\infty}_2(z,T_\infty) &= J_2(z,0,\rho_{S_\infty})\vee J_2(z,\rho^+_{T_\infty},\rho_{S_\infty})=f_2(z)\vee M\delta_2(1)^2\\
&=f_2(z)< M\delta_2(2) = J_2(z,\rho^+_\emptyset,\rho_{S_\infty}) \le U^{S_\infty}_2(z,\emptyset). 
\end{align*}
As $T_\infty\notin \widehat \cE^{S_\infty}_2$, $(S_\infty, T_\infty)= (\{x_0,x_1,x_2,\dotso\},\{y,z\})$ is not a sharp inter-personal equilibrium. 
\end{example}


\subsection{General Existence of Sharp Inter-Personal Equilibria}
In view of Example \ref{eg2}, the soft inter-personal equilibrium constructed in Theorem~\ref{t2} may not be a sharp one. That is to say, the general existence of a sharp inter-personal equilibrium is still in question. To resolve this, we impose appropriate regularity of $X$.

\begin{assumption}\label{asm:density}
$X$ has transition densities $(p_t)_{t\ge 1}$ with respect to a measure $\mu$ on $(\X, \B)$. That is, for each $t=1,2,...$, $p_t: \X\times \X\to \R_+$ is a Borel measurable function 
such that
\[
\P_x(X_t\in A) = \int_A p_t(x,y) \mu(dy),\quad \forall x\in\X\ \hbox{and}\ A\in\B.
\]
\end{assumption}

\begin{remark}
When $\X$ is at most countable, Assumption 4.1 is trivially satisfied. When $\X$ is uncountable, the literature is focused on the case $\X=\R^d$ for some $d\ge 1$. In this case, many discrete-time Markov processes $X$ fulfill Assumption 4.1 (with $\mu$ being the Lebesgue measure). This includes, particularly, $X$ defined by the formula
$X_{t+1} := G(X_t,Z_t)$, $t\in \Z_+$, 
where $G$ is a Borel measurable function and $Z_t$ is a random variable independent of $X$ such that $G(x,Z_t)$ admits a probability density function for all $x\in\X$. This formula is commonly used in practical simulation of Markov processes; see e.g. \cite[Section 3]{Sart14} and \cite[Section 5]{AL11}.
\end{remark}

The next result, as a direct consequence of \cite[Lemma 6.5]{Kosorok-book-08}, will also play a crucial role.

\begin{lemma}\label{lem:minorant}
Let $\mu$ be a measure on $(\X,\B)$. For any $A\subseteq \X$, there exists a maximal Borel minorant of $A$ under $\mu$, defined as a set $A^\mu\in\B$ with $A^\mu\subseteq A$ such that for any $A'\in\B$ with $A'\subseteq A$, $\mu(A'\setminus A^\mu)=0$. 
\end{lemma}


Now, we are ready to present the general existence of a sharp inter-personal equilibrium. 

\begin{theorem}\label{t3}
Suppose Assumption~\ref{asm:density} holds. For each $i\in\{1,2\}$, assume $f_i\le h_i\le g_i$ and \eqref{supermartingale}. 
Then, there exists a sharp inter-personal equilibrium.
\end{theorem}

\begin{proof}
Consider the set
$$A:=\{(S,T)\in\cE : \Gamma_1(T)\supseteq S\ \hbox{and}\ \Gamma_2(S)\subseteq T\}.$$
By \thmref{t2}, $A\neq\emptyset$. Now, define a partial order on $A$ as follows: for any $(S,T), (S',T')\in A$, 
\begin{equation}\label{p order}
(S,T)\succeq (S',T')\quad \text{if}\ S\supseteq S'\ \text{and}\ T\subseteq T'.
\end{equation}
{\bf Step 1:} {\it Showing that every totally ordered subset of $A$ has an upper bound in $A$.}\\
Let $(S_\alpha,T_\alpha)_{\alpha\in I}$ be a subset of $A$ that is totally ordered. Set $S_0:=\bigcup_{\alpha\in I} S_{\alpha}$ and $T_0:=\bigcap_{\alpha\in I} T_{\alpha}$. 
Recall the measure $\mu$ in Assumption~\ref{asm:density}. By Lemma~\ref{lem:minorant}, there exists a maximal Borel minorant of $T_0$ under $\mu$, which will be denoted by $T_0^\mu$. 
For any $T\in\B$ with $T\subseteq T_0$, since $\mu(T\setminus T_0^{\mu})=0$, we deduce from Assumption~\ref{asm:density} that $\P_x(X_t\in T\setminus T_0^{\mu}) = \int_{T\setminus T_0^{\mu}} p_t(x,y) \mu(dy)=0$ for all $x\in\X$ and $t\in\N$. It follows that 
\begin{equation}\label{polar}
\P_x(X_t\in T\setminus T_0^{\mu}\ \hbox{for some $t\in\N$}) =0\quad \forall x\in\X,\quad \hbox{whenever $T\in\B$ and $T\subseteq T_0$}. 
\end{equation}
On the other hand, observe that 
\begin{equation}\label{sub, sup}
\Gamma_1(T)\supseteq S_0\ \hbox{for any $T\in\B$ with $T\subseteq T_0$},\quad \Gamma_2(S)\subseteq T_0\ \hbox{for any $S\in\B$ with $S\supseteq S_0$}.
\end{equation}
Indeed, for any $T\in\B$ with $T\subseteq T_0=\bigcap_{\alpha\in I} T_\alpha$, by Corollary~\ref{l8} and the definition of $A$, $\Gamma_1(T)\supseteq \Gamma_1(T_\alpha)\supseteq S_\alpha$ for all $\alpha\in I$, which implies $\Gamma_1(T)\supseteq S_0$. Similarly, for any $S\in\B$ with $S\supseteq S_0=\bigcup_{\alpha\in I} S_\alpha$, by Corollary~\ref{l8} and the definition of $A$, $\Gamma_2(S)\subseteq \Gamma_2(S_\alpha)\subseteq T_\alpha$ for all $\alpha\in I$, which implies $\Gamma_2(S)\subseteq T_0$.   

Now, define
\[
S_1:=\Gamma_1(T_0^\mu)\supseteq S_0\quad \hbox{and}\quad T_1:=\Gamma_2(S_1)\subseteq T_0. 
\]
Note that the first inclusion follows from $T_0^\mu\in\B$, $T_0^\mu\subseteq T_0$, and \eqref{sub, sup}. As $T_0^\mu\in\B$ implies $S_1:=\Gamma_1(T_0^\mu)\in\B$, we deduce from $S_1\in\B$, $S_1\supseteq S_0$, and \eqref{sub, sup} that the second inclusion above holds. With $T_1\in\B$ (thanks to $S_1\in\B$) and $T_1\subseteq T_0$, \eqref{polar} gives $\P_x(X_t\in T_1\setminus T_0^{\mu}\ \hbox{for some $t\in\N$}) =0$ for all $x\in\R^d$. This readily implies 
\begin{equation}\label{123'}
\rho_{T_1\cup T_0^\mu} = \rho_{T_0^\mu}\quad\P_x\hbox{-a.s.},\quad \hbox{for}\ x\notin T_1\cup T^\mu_0. 
\end{equation}

We claim that $\Gamma_1(T_1\cup T_0^\mu) = \Gamma_1(T_0^\mu)$. 
By the definition of $\Gamma_1$ in \eqref{S_infty}, it suffices to show that $S^n_1(T_1\cup T^\mu_0)= S^n_1(T^\mu_0)$ for all $n\in \N$. 
First, as $T_1 = \Gamma_2(S_1)$, the last assertion of Proposition~\ref{l3} implies $ T_1\cap S_1=\emptyset$. With $S_1=\Gamma_1(T^\mu_0) = \bigcup_{n\in\N} S_1^n(T^\mu_0)$, we obtain 
\begin{equation}\label{dudu}
T_1\cap S_1^n(T^\mu_0)=\emptyset,\quad \forall n\in\N. 
\end{equation}
Now, for $n=1$, \eqref{Phi inclusion} implies $S_1^1(T_1\cup T^\mu_0)= \Phi_1^{T_1\cup T^\mu_0}(\emptyset)\subseteq\Phi_1^{T^\mu_0}(\emptyset)=S_1^1(T^\mu_0)$. For any $x\in S_1^1(T^\mu_0)$, due to $T_0^\mu\cap S_1^1(T^\mu_0)=\emptyset$ (by Proposition~\ref{l3}) and \eqref{dudu}, we must have $x\notin T_1\cup T^\mu_0$. Observe that
\begin{equation}\label{first}
J_1(x,0,\rho_{T_1\cup T^\mu_0})= J_1(x,0,\rho_{T^\mu_0}) > V^{T^\mu_0}_1(x,\emptyset)=V^{T_1\cup T^\mu_0}_1(x,\emptyset), 
\end{equation}
where the first and third equalities follow from \eqref{123'} and the inequality stems from the definition of $S_1^1(T^\mu_0)=\Phi_1^{T^\mu_0}(\emptyset)$ in \eqref{Phi_i^T}. This shows that $x\in \Phi_1^{T_1\cup T^\mu_0}(\emptyset)=S_1^1(T_1\cup T^\mu_0)$. Hence, we obtain $S_1^1(T^\mu_0)\subseteq S_1^1(T_1\cup T^\mu_0)$ and thus conclude $S_1^1(T_1\cup T^\mu_0)=S_1^1(T^\mu_0)$. Suppose that $S_1^k(T_1\cup T^\mu_0)=S_1^k(T^\mu_0)$ for some $k\ge 1$. By \eqref{Phi inclusion} again, $S^{k+1}_1(T_1\cup T^\mu_0)\subseteq S^{k+1}_1(T^\mu_0)$. Fix $x\in S_1^{k+1}(T^\mu_0)$. By using Proposition~\ref{l3} and \eqref{123'} as above, we get $x\notin T_1\cup T^\mu_0$. If $x\in S_1^{k}(T_1\cup T^\mu_0)$, then $x\in S_1^{k+1}(T_1\cup T^\mu_0)$ trivially, by the definition of $S_1^{k+1}(T_1\cup T^\mu_0)$ in \eqref{S_n}. If $x\notin S_1^{k}(T_1\cup T^\mu_0)=S_1^k(T^\mu_0)$, then by the definition of $S_1^{k+1}(T^\mu_0)$,
\[
J_1(x,0,\rho_{T^\mu_0}) > V^{T^\mu_0}_1(x,S_1^k(T^\mu_0)). 
\]
By \eqref{123'} and the above inequality, we may argue similarly as in \eqref{first} to get
\begin{align*}
J_1(x,0,\rho_{T_1\cup T^\mu_0})= J_1(x,0,\rho_{T^\mu_0}) > V^{T^\mu_0}_1(x,S_1^k(T^\mu_0))&=V^{T^\mu_0}_1(x,S_1^k(T_1\cup T^\mu_0))\\
&=V^{T_1\cup T^\mu_0}_1(x,S_1^k(T_1\cup T^\mu_0)),
\end{align*}
where the second equality is due to $S^k(T^\mu_0)=S^k(T_1\cup T^\mu_0)$. It follows that $x\in S^{k+1}(T_1\cup T_0^\mu)$. Hence, we obtain $S^{k+1}(T^\mu_0)\subseteq S^{k+1}(T_1\cup T^\mu_0)$ and thus conclude $ S^{k+1}(T_1\cup T^\mu_0)=S^{k+1}(T^\mu_0)$. By induction, we have established $S^{n}(T^\mu_0)= S^{n}(T_1\cup T^\mu_0)$ for all $n\in\N$, as desired. 

By Corollary~\ref{l8} and $\Gamma_1(T_1\cup T_0^\mu) = \Gamma_1(T_0^\mu)$,   
\begin{align*}
S_2 := \Gamma_1(T_1) \supseteq \Gamma_1(T_1\cup T_0^\mu) = \Gamma_1(T_0^\mu)=S_1\quad\hbox{and}\quad T_2 := \Gamma_2(S_2) \subseteq \Gamma_2(S_1) =T_1.
\end{align*}
Now, by defining $S_{n+1}:=\Gamma_1(T_n)$ and $T_{n+1}:=\Gamma_2(S_{n+1})$ for all $n\ge 3$, we can follow the same argument in the proof of Theorem \ref{t2} to show that $(S_n)$ is nondecreasing, $(T_n)$ is nonincreasing, and $(S_\infty,T_\infty)\in A$ with $S_\infty:=\bigcup_n S_n$ and $T_\infty:=\bigcap_n T_n$. By construction, $S_\infty\supseteq S_0\supseteq S_\alpha$ and $T_\infty\subseteq T_0\subseteq T_\alpha$ for all $\alpha\in I$. Hence, $(S_\infty,T_\infty)\in A$ is an upper bound for $(S_\alpha,T_\alpha)_{\alpha\in I}$.
\vspace{0.05in}\\
{\bf Step 2:} {\it Applying Zorn's lemma.}\\
As every totally ordered subset of $A$ is shown to have an upper bound in $A$, Zorn's lemma implies that there exists a maximal element in $A$, denoted by $(\bar S,\bar T)\in A$. We claim that $(\bar S,\bar T)\in\widehat\cE$. Set $S_0:=\bar S$, $T_0 := \bar T$, and define
\[
S_{n+1}:=\Gamma_1(T_n)\quad \hbox{and}\quad T_{n+1}:=\Gamma_2(S_{n+1})\quad \forall n\ge 0. 
\]
Thanks to $\Gamma_1(T_0)\supseteq S_0$ and $\Gamma_2(S_0)\subseteq T_0$ (as $(S_0,T_0)=(\bar S, \bar T) \in A$), we may apply Corollary~\ref{l8} recursively to show that $(S_n)$ is nondecreasing and $(T_n)$ is nonincreasing. Then, by the same argument in the proof of Theorem \ref{t2}, we obtain $(S_\infty,T_\infty)\in A$ with $S_\infty:=\bigcup_n S_n$ and $T_\infty:=\bigcap_n T_n$. By construction, $S_\infty\supseteq S_0= \bar S$ and $T_\infty\subseteq T_0=\bar T$. But since $(\bar S,\bar T)$ is a maximal element of $A$ (under the partial order \eqref{p order}), we must have $S_\infty=S_0$ and $T_\infty=T_0$. This in particular implies $S_1= S_0$ and $T_1=T_0$, so that
\[
\Gamma_1(\bar T) = \Gamma_1(T_0) = S_1 =S_0 = \bar S\quad \hbox{and}\quad \Gamma_2(\bar S) = \Gamma_2(S_0) = \Gamma_2(S_1)= T_1 =T_0 = \bar T. 
\]
By Lemma~\ref{lem:sufficient}, this readily implies $(\bar S,\bar T)\in\widehat\cE$. 
\end{proof}

\begin{remark}
In view of \eqref{J}-\eqref{F}, the condition $f_i \le h_i \le g_i$ (in Theorems~\ref{t2} and \ref{t3}) encourages each player to wait/continue until the other player stops, so as to obtain a larger reward. Consequently, each player faces the tradeoff between the potential (generous) gain from outlasting the other player and the cost of waiting that enlarges with time (due to discounting and possible loss of opportunity). That is, our Dynkin game exemplifies the ``war of attrition'' in game theory. The negotiation example in Section \ref{sec:application} below well demonstrates this ``war'': Each firm intends to wait until the other firm gives in so as to seal the best deal, while subject to the impact of discounting and the varying cost of project initiation.
\end{remark}

\begin{remark}\label{rem:randomized}
In a classical (time-consistent) nonzero-sum Dynkin game, the condition $f_i\le h_i\le g_i$ ensures that a Nash equilibrium, as a tuple of pure stopping times $(\tau^*,\sigma^*)$, exists; see e.g., \cite{HZ09}. Without the condition $f_i\le h_i\le g_i$, a Nash equilibrium $(\tau^*,\sigma^*)$ need not exist, as shown in \cite{LS13}. One needs to consider randomized strategies to possibly establish the existence of a Nash equilibrium, as a tuple of randomized strategies. Still, in some cases, only an $\eps$-Nash equilibrium is known to exist; see e.g., \cite{SS04, F05, LS13}. 

In this paper, as we assume $f_i\le h_i\le g_i$ (cf. Theorems~\ref{t2} and \ref{t3}), our focus on pure strategies is consistent with the literature. If we drop the condition $f_i\le h_i\le g_i$, many arguments will no longer hold and we expect the use of randomized strategies indispensable. Randomized strategies for time-inconsistent stopping problems have recently been proposed and analyzed by \cite{BZZ19} in discrete time and by \cite{CL20} in continuous time. It is of interest as future research to modify their definitions and allow for randomized strategies in our Dynkin game.
\end{remark}


\subsection{Discussion on the Supermartingale Condition}\label{subsec:supermartingale} 
It is worth noting that while the supermartingale condition \eqref{supermartingale} is required in Theorems~\ref{t2} and \ref{t3}, it does not play a role in Theorem~\ref{t1}. Because the one-player iterative procedure \eqref{S_n} is by construction monotone, it  converges without the need of any other condition. It is much more complicated for the two-player alternating iterative procedure \eqref{alternating} to converge. The monotonicity of \eqref{S_n} only ensures that each iteration (performed by one of the two players) converges to a stopping policy, but says nothing about whether the two resulting sequences of policies (one sequence for each player) will actually converge. It is the supermartingale condition \eqref{supermartingale} that brings about the monotonicity for these two sequences of policies (on strength of Corollary~\ref{l8}), leading to an inter-personal equilibrium between the two players. 

When \eqref{supermartingale} fails, the monotonicity in Corollary~\ref{l8} no longer holds in general, and there may exist {\it no} inter-personal equilibrium, soft or sharp. To demonstrate this, consider a three-state model
\begin{equation}\label{three state}
\X=\{a,b,c\}\quad \hbox{with}\quad \P_x(X_1=y)
\begin{cases}
=0,\quad \hbox{for}\ (x,y)=(a,c),\\
>0,\quad \hbox{otherwise}. 
\end{cases}
\end{equation}
Given $M>0$, define the payoff functions by
\begin{align}\label{f,g}
f_1(a)&=1,& g_1(a)&=M^2,& f_2(a)&=M,& g_2(a)&=M+1,\notag\\
f_1(b)&=M, & g_1(b)&=M+1,& f_2(b)&=1,& g_2(b)&=2,\\
f_1(c)&=1,& g_1(c)&=2,& f_2(c)&=M^2,& g_2(c)&=M^2+1,\notag
\end{align}
and
\begin{equation}\label{h}
h_i(x)=\frac{1}{2}(f_i(x)+g_i(x)),\quad \forall x\in\X\ \hbox{and}\ i\in\{1,2\}.
\end{equation}

\begin{proposition}\label{prop:no weak E} 
Under \eqref{three state}, \eqref{f,g}, and \eqref{h}, as $M>0$ is large enough, \eqref{supermartingale} is violated and there exists no soft inter-personal equilibrium.
\end{proposition}

\begin{proof}
Let $p_{xy}:=\P_x(X_1=y)$ for all $x,y\in \X$. Then,
\[
\E_c[\delta_1(1)g_1(X_1)] = \delta_1(1) \left(p_{ca} M^2 + p_{cb} (M+1)+p_{cc}\cdot2\right) > 2 = g_1(c),
\]
where the inequality holds as $M>0$ is large enough. This readily shows that \eqref{supermartingale} is violated.

For any $S,T\in \B=2^\X$ and $x\in\X$, we deduce from the definitions of $f_i$, $h_i$, $g_i$, $i\in\{1,2\}$, that 
for each $i\in\{1,2\}$, 
\begin{equation}\label{J_i expression}
J_i(x,\rho^+_S,\rho_T)= k_0 + k_1 M + k_2 M^2,\quad \hbox{for some}\ k_0\in [0,2]\ \hbox{and}\ k_1, k_2\in [0,1].
\end{equation}
Moreover,
\begin{equation}\label{k_0>0}
\rho_T>0 \iff k_0<2,\ k_1<1,\ \hbox{and}\ k_2<1. 
\end{equation}
By \eqref{J_i expression} and \eqref{k_0>0}, it can be checked that 
\begin{equation}\label{identities}
\begin{split}
&\cE_2^\X = \cE_2^{\{a,c\}} = \cE_2^{\{c\}}= \{\emptyset\}\ \ \hbox{but}\ \ \cE_1^\emptyset=\{\{b\}\},\\
& \cE_2^{\{a,b\}} =\cE_2^{\{a\}} = \cE_2^\emptyset= \{\{c\}\}\ \ \hbox{but}\ \ \cE_1^{\{c\}} = \{\{b\}\},\\
& \cE_2^{\{b,c\}} = \{\{a\}\}\ \ \hbox{but}\ \ \cE_1^{\{a\}} = \emptyset,\\
& \cE_2^{\{b\}} = \{\{a,c\}\}\ \ \hbox{but}\ \ \cE_1^{\{a,c\}}=\emptyset.
\end{split}
\end{equation}
This readily shows that there exists no $(S,T)\in 2^\X\times 2^\X$ such that $S\in\cE_1^T$ and $T\in\cE_2^S$, i.e. there exists no soft inter-personal equilibrium. 

In the following, we will show the derivation of $\cE_1^{\{c\}} = \{\{b\}\}$ in detail, while all other identities in \eqref{identities} can be proved in a similar manner. 
It follows directly from the definitions of $f_1$ and $g_1$ that $J_1(b,0,\rho_{\{c\}}) = f_1(b) =M$ and $J_1(b,\rho^+_\emptyset,\rho_{\{c\}})= k_0<2$. Hence, with $M>0$ large enough,
\[
b\in \left\{x\in\X: J_1(x,0,\rho_{\{c\}})> J_1(x,\rho^+_\emptyset,\rho_{\{c\}})\right\}= \Theta_1^{\{c\}}(\emptyset),
\]
which implies $\Theta_1^{\{c\}}(\emptyset)\neq \emptyset$. Also, as $J_1(b,\rho^+_{\{a\}},\rho_{\{c\}})=k_0<2$, we can similarly conclude that with $M>0$ large enough,
\[
b\in \left\{x\in\{b,c\}: J_1(x,0,\rho_{\{c\}})> J_1(x,\rho^+_{\{a\}},\rho_{\{c\}})\right\}\subseteq  \Theta_1^{\{c\}}(\{a\}),
\]
which implies $\Theta_1^{\{c\}}(\{a\})\neq \{a\}$. Note that
\begin{equation}\label{c}
J_1(c,0,\rho_{\{c\}}) = h_1(c)< g_1(c) = J_1(c,\rho^+_S,\rho_{\{c\}}),\quad \forall S\in 2^\X.
\end{equation}
By taking $S=\{c\}$, $S=\{b,c\}$, and $S=\X$ in \eqref{c}, we immediately see that $c\notin \Theta^{\{c\}}_1(\{c\})$, $c\notin \Theta^{\{c\}}_1(\{b,c\})$, and $c\notin \Theta^{\{c\}}_1(\X)$. We therefore conclude $\Theta^{\{c\}}_1(\{c\})\neq\{c\}$, $\Theta^{\{c\}}_1(\{b,c\})\neq\{b,c\}$, and $\Theta^{\{c\}}_1(\X)\neq\X$. Now, observe that $J_1(b, \rho^+_{\{b\}},\rho_{\{c\}})$ and $J_1(a, \rho^+_{\{b\}},\rho_{\{c\}})$ are both of the form $k_0+k_1 M$ with $k_1<1$ (recall \eqref{k_0>0}). Thus, with $M>0$ large enough,
\begin{align}
J_1(b,0,\rho_{\{c\}}) &= f_1(b)= M > k_0+k_1 M = J_1(b, \rho^+_{\{b\}},\rho_{\{c\}}),\\
J_1(a,0,\rho_{\{c\}}) &= f_1(a)=1< k_0+k_1 M = J_1(a, \rho^+_{\{b\}},\rho_{\{c\}}).\label{a}
\end{align}
This, together with \eqref{c}, implies $\Theta^{\{c\}}_1(\{b\})=\{b\}$. Since $J_1(a,\rho^+_{\{a,b\}},\rho_{\{c\}})$ (resp. $J_1(a,\rho^+_{\{a,c\}},\rho_{\{c\}})$) is also of the form $k_0+k_1 M$, the inequality in \eqref{a} indicates that with $M>0$ large enough, $a\notin \Theta^{\{c\}}_1(\{a,b\})$ (resp. $a\notin \Theta^{\{c\}}_1(\{a,c\})$). Hence, we conclude $\Theta^{\{c\}}_1(\{a,b\})\neq\{a,b\}$ and $\Theta^{\{c\}}_1(\{a,c\})\neq\{a,c\}$. In view of the above derivations, $\{b\}$ is the only intra-personal equilibrium for Player 1 w.r.t. Player 2's policy $\{c\}$, i.e. $\cE_1^{\{c\}}=\{\{b\}\}$. 
\end{proof}

\begin{remark}
In \eqref{identities}, $\cE_i^T$ is a singleton for all $i\in\{1,2\}$ and $T\in 2^\X$. The single element in $\cE_i^T$ is trivially the optimal intra-personal equilibrium, which can be recovered by $\Gamma_i(T)$ in \eqref{S_infty} thanks to Theorem \ref{t1}. In other words, \eqref{identities} implies
\begin{equation}\label{identities'}
\begin{split}
&\Gamma_2(\X) = \Gamma_2(\{a,c\}) =\Gamma_2(\{c\}) = \emptyset,\quad \Gamma_2(\{a,b\}) = \Gamma_2(\{a\}) =\Gamma_2(\emptyset) = \{c\},\\
& \Gamma_2(\{b,c\})=\{a\},\quad \Gamma_2(\{b\})=\{a,c\},\\
& \Gamma_1(\emptyset)= \Gamma_1(\{c\})=\{b\},\quad \Gamma_1(\{a\}) = \Gamma_1(\{a,c\})=\emptyset.
\end{split}
\end{equation}
This clearly shows that the monotonicity of $T\mapsto \Gamma_i(T)$ 
fails: Despite the inclusion $\{c\}\subseteq\{b,c\}\subseteq \X$, we have $\Gamma_2(\{c\}) =\Gamma_2(\X) = \emptyset\subseteq\{a\}=\Gamma_2(\{b,c\})$.
\end{remark}

\begin{remark}
Another way to interpret \eqref{identities'} is that the alternating iterative procedure \eqref{alternating} will never converge, failing to provide any soft inter-personal equilibrium. Specifically, \eqref{identities'} indicates that the alternating iterations will always lead to loops, as listed below where Player 1's stopping policies are underlined and Player 2's stopping policies are double underlined.
\vspace{0.05in}

1. $\uline{\emptyset}\to \uuline{\{c\}}\to\uline{\{b\}}\to\uuline{\{a,c\}}\to\uline{\emptyset}\to...$.

2. $\uline{\{a\}}\to\uuline{\{c\}}\to\uline{\{b\}}\to\uuline{\{a,c\}}\to\uline{\emptyset}\to\uuline{\{c\}}\to...$.

3. $\uline{\{b\}}\to\uuline{\{a,c\}}\to\uline{\emptyset}\to\uuline{\{c\}}\to\uline{\{b\}}\to...$.

4. $\uline{\{c\}}\to\uuline{\emptyset}\to\uline{\{b\}}\to\uuline{\{a,c\}}\to\uline{\emptyset}\to\uuline{\{c\}}\to\uline{\{b\}}\to...$.

5. $\uline{\{a,b\}}\to\uuline{\{c\}}\to\uline{\{b\}}\to\uuline{\{a,c\}}\to\uline{\emptyset}\to\uuline{\{c\}}\to...$.

6. $\uline{\{a,c\}}\to\uuline{\emptyset}\to\uline{\{b\}}\to\uuline{\{a,c\}}\to\uline{\emptyset}\to\uuline{\{c\}}\to\uline{\{b\}}\to...$.

7. $\uline{\{b,c\}}\to\uuline{\{a\}}\to\uline{\emptyset}\to\uuline{\{c\}}\to\uline{\{b\}}\to\uuline{\{a,c\}}\to\uline{\emptyset}\to...$.

8. $\uline{\{a,b,c\}}\to\uuline{\emptyset}\to\uline{\{b\}}\to\uuline{\{a,c\}}\to\uline{\emptyset}\to\uuline{\{c\}}\to\uline{\{b\}}\to...$.
\end{remark}

\begin{remark}
As $\delta$ is not specified in Proposition~\ref{prop:no weak E}, the result admits an interesting implication for classical (time-consistent) nonzero-sum Dynkin games. To see this, let $\delta$ be an exponential discount function so that there is no time inconsistency. Proposition~\ref{prop:no weak E} shows that even when the state process $X$ is time-homogeneous and payoff functions are as simple as \eqref{f,g}-\eqref{h}, the Dynkin game has no ``time-homogeneous'' Nash equilibrium---a Nash equilibrium as a tuple of two stopping regions $(S,T)$, one for each player. Indeed, if such a Nash equilibrium existed, it would be a sharp inter-personal equilibrium under Definition~\ref{def:sharp}. If a Nash equilibrium in fact exists, it must be of a more complicated form (which is also suggested by the constructions in \cite{HZ09, LS13}). 
\end{remark}



\section{Application: Negotiation with Diverse Impatience}\label{sec:application}
In this section, we apply our theoretic results in Section~\ref{sec:existence} to a two-player real options valuation problem. The vast literature on real options, see e.g. \cite{MS86, Dixit94, Smith95} among many others, focuses on a single firm's corporate decision making, particularly the optimal timing of a project's initiation. In contrast to this, we will study two firms' joint decision making on their cooperation to initiate a project together, embedding real options valuation in a nonzero-sum Dynkin game.  

Consider two firms who would like to cooperate to initiate a new project, such as entering a new market or developing a new product. Each firm has a proprietary skill/technology, so that only when they cooperate can the project be successfully carried out. Once the project is initiated, it will generate a fixed total revenue $R>0$. The cost of initiation $X$, on the other hand, evolves stochastically and is modeled by a discrete-time binomial structure as follows: There exist $u>1$ and $p\in(0,1)$ such that $X$ takes values in 
\begin{equation}
\X = \{u^i: i = 0, \pm1, \pm2, \dotso\}\label{X app}
\end{equation}
and satisfies
\[
\P_x(X_{1}/x = u)=p\quad \hbox{and}\quad \P_x(X_{1}/x = 1/u)=1-p,\quad \forall x\in\X.
\]
Assume additionally that $X$ is a submartingale, which corresponds to the condition 
$p\ge \frac{1}{u+1}$.
That is, the cost $X$ has a tendency to increase over time, which incentivizes the two firms to strike a deal of cooperation sooner than later.  

In negotiating such a deal, each firm, leveraging on its proprietary skill/technology, insists on taking a fixed (risk-free) larger share 
\begin{equation*}
N\in(R/2,R)
\end{equation*} 
of the total revenue $R>0$, while demanding the other firm to take the smaller share 
\begin{equation*}
K:= R- N \in (0,R/2)
\end{equation*} 
of revenue and additionally incur the stochastic (risky) cost $X$. Each firm either waits until the other gives in and takes the larger payoff $N$, or gives in to the other and takes the smaller payoff $(K-X_\tau)^+$, where $\tau$ denotes the firm's (random) time to give in. This can be formulated in our Dynkin game framework as
\[
f_1(x)=f_2(x)=(K-x)^+\quad \hbox{and}\quad g_1(x) = g_2(x)=N,\quad \forall x\in\X. 
\]
If the two firms happen to give in at the same time, they realize that both of them cannot endure any delay of a deal, and will quickly agree on a deal that is more mutually beneficial. This corresponds to the requirement $f_i\le h_i\le g_i$, $i\in\{1,2\}$. In addition, we model the time preferences of the firms using the hyperbolic discount function, i.e. for $i\in\{1,2\}$,
\[
\delta_i(t) = \frac{1}{1+\beta_i t},
\] 
where $\beta_i >0$ is a constant that represents the level of impatience of Firm $i$.  

To facilitate the investigation of inter-personal equilibria between the two firms, we introduce a random walk $Y$ defined on some probability space $(\bar\Omega, \bar\F, P)$ such that 
\[
P(Y_{t+1}-Y_t = 1) = p\quad \hbox{and}\quad P(Y_{t+1}-Y_t =-1) = 1- p,\quad\forall t\in\Z_+.
\]
Consider
\[
\xi:=\inf\{t\ge 0: Y_t =0\}
\]
and define, for each $i\in\{1,2\}$,
\begin{equation}\label{alpha_i^n}
\alpha^n_i:=E^n\left[\frac{1}{1+\beta_i \xi}\right]\ \ \forall n\in\N,
\end{equation}
where $E^n$ denotes the expectation under $P$ conditioned on $Y_0 = n$. Note that $\alpha_n$, $n\in\N$, can be computed explicitly. For instance,
\begin{equation}\label{alpha 1}
\alpha^1_i=\sum_{k=1}^\infty\frac{\binom{2k-1}{k}p^{k-1}(1-p)^k}{2k-1}\cdot\frac{1}{1+\beta_i(2k-1)}. 
\end{equation}

\begin{lemma}\label{lem:app}
For $i\in\{1,2\}$, $\Gamma_i(\emptyset) = (0,y_i^*]\cap\X$, where 
\begin{equation}\label{y^*}
y_i^*:= \min\bigg\{\bigg[\frac{1-\alpha_i^1}{u-\alpha_i^1}K,\infty\bigg)\cap\X\bigg\}.
\end{equation}
\end{lemma}

\begin{proof}
Observe from \eqref{J} and \eqref{F} that
\[
J_i(x,\tau,\rho_\emptyset) = \E_x[F_i(\tau,\rho_\emptyset)] = \E_x[\delta_i(\tau)f_i(X_\tau)] = \E_x\bigg[\frac{(K-X_\tau)^+}{1+\beta_i \tau}\bigg],\quad \forall x\in\X. 
\]
Hence, the one-player stopping analysis in \cite[Section 5]{HZ19} applies to our current setting. The same arguments therein (particularly \cite[Proposition 5.5]{HZ19}) show that $(0,y_i^*]\cap\X$, with $y_i^*$ given as in \eqref{y^*}, is Player $i$'s unique optimal intra-personal equilibrium w.r.t. $\emptyset$, i.e. $\widehat \cE^\emptyset_i =\{(0,y_i^*]\cap\X\}$. As $\Gamma_i(\emptyset)$ belongs to $\widehat \cE^\emptyset_i$ by Theorem~\ref{t1}, it must coincide with $(0,y_i^*]\cap\X$.  
\end{proof}


With the aid of Lemma~\ref{lem:app}, we will show that the alternating iterative procedure \eqref{alternating} always leads to a sharp inter-personal equilibrium. Let us divide our investigation into two cases, depending on the impatience levels of the two firms: $\beta_1\le \beta_2$ (Proposition~\ref{prop:1<2}) and $\beta_1> \beta_2$ (Proposition~\ref{prop:1>2}). 

\begin{proposition}\label{prop:1<2}
Suppose $\beta_1\le \beta_2$. Then, the alternating iterative procedure \eqref{alternating} terminates after one iteration, and gives a sharp inter-personal equilibrium. That is, 
\[
(S_\infty,T_\infty)=(S_0,T_0) = (\emptyset,(0,y_2^*]\cap\X) \in \widehat\cE. 
\]
\end{proposition}

\begin{proof}
Note from \eqref{alpha 1} that $\beta_1\le \beta_2$ implies $\alpha_1^1\ge \alpha_2^1$. This in turn yields $y_1^*\le y_2^*$, in view of \eqref{y^*}. 
Following \eqref{alternating}, we have $S_0 :=\emptyset$ and $T_0 := \Gamma_2(S_0)= \Gamma_2(\emptyset) = (0,y_2^*]\cap\X$, where the last equality is due to Lemma~\ref{lem:app}. Now, in view of Corollary~\ref{l8}, Lemma~\ref{lem:app}, and Proposition~\ref{l3}, 
\begin{equation}\label{use 3}
\begin{split}
&\Gamma_1(T_0) \subseteq \Gamma_1(\emptyset) = (0,y_1^*]\cap\X,\\
&\Gamma_1(T_0)\cap \big( (0,y_2^*]\cap\X\big) =\Gamma_1(T_0)\cap T_0=\emptyset. 
\end{split}
\end{equation}
As $y_1^*\le y_2^*$, the above two relations entail $\Gamma_1(T_0) = \emptyset =S_0$. This, together with $\Gamma_2(S_0)=T_0$, shows that $(S_0,T_0)\in \widehat\cE$, thanks to Lemma~\ref{lem:sufficient}.
\end{proof}

\begin{remark}\label{rem:1>2}
For the case $\beta_1\ge \beta_2$, by starting with $T_0=\emptyset$ and switching the roles of $S_n$ and $T_n$ in \eqref{alternating}, we may repeat the same arguments in the above proof to show that $T_0$ and $S_0 := \Gamma_1(T_0) = (0,y^*_1]\cap\X$ satisfy $\Gamma_2(S_0) =\emptyset=T_0$, which then yields $(S_0,T_0)=((0,y^*_1]\cap\X,\emptyset)\in\widehat\cE$. 
\end{remark}

\begin{proposition}\label{prop:1>2}
Suppose $\beta_1> \beta_2$. Then, the alternating iterative procedure \eqref{alternating} falls into the following three cases.
\begin{itemize}
\item {\bf Case 1:} $S_1=\emptyset$. Then $(S_\infty,T_\infty)= (\emptyset,(0,y_2^*]\cap\X) \in \widehat\cE$.
\item {\bf Case 2:} $S_1\neq\emptyset$ and $T_n =\emptyset$ for some $n>1$. Then $(S_{\infty},T_{\infty}) = ((0,y_1^*]\cap\X, \emptyset) \in \widehat\cE$.
\item {\bf Case 3:} $S_1\neq\emptyset$ and $T_n \neq\emptyset$ for all $n>1$. 
\begin{itemize}
\item If \eqref{alternating} terminates in finite steps, then there exist $y^*_2<c\le y^*_1$ and $0<d\le y^*_2$ such that 
$
(S_\infty,T_\infty)= ([c,y_1^*]\cap\X, (0,d]\cap\X) \in \widehat\cE.  
$
\item If \eqref{alternating} continues indefinitely, $(S_\infty,T_\infty)= ((0,y_1^*]\cap\X, \emptyset) \in \widehat\cE$. 
\end{itemize}
\end{itemize}
\end{proposition}

\begin{proof}
Note from \eqref{alpha 1} and \eqref{y^*} that $\beta_1> \beta_2$ implies $y_1^*\ge y_2^*$. If $y_1^*= y_2^*$ (which happens when $\beta_1-\beta_2>0$ is sufficiently small), the same arguments in the proof of Proposition~\ref{prop:1<2} apply, and we end up with Case 1. In the rest of the proof, we assume $y_1^*> y_2^*$. Following \eqref{alternating}, $S_0 :=\emptyset$ and $T_0 := \Gamma_2(S_0)= \Gamma_2(\emptyset) = (0,y_2^*]\cap\X$, where the last equality is due to Lemma~\ref{lem:app}. If $S_1:=\Gamma_1(T_0)=\emptyset=S_0$, \eqref{alternating} terminates with $\Gamma_1(T_0)=S_0$ and $\Gamma_2(S_0)=T_0$. Hence, $(S_\infty,T_\infty)=(S_0,T_0) =  (\emptyset,(0,y_2^*]\cap\X) \in \widehat\cE$ by Lemma~\ref{lem:sufficient}. Alternatively, if $S_1\neq\emptyset$ but $T_n =\emptyset$ for some $n>1$, following the argument in Remark~\ref{rem:1>2} gives $(S_{n+1},T_{n+1}) = ((0,y_1^*]\cap\X, \emptyset) \in \widehat\cE$.

It remains to deal with the situation where $S_1\neq \emptyset$ and $T_n\neq\emptyset$ for all $n\ge 1$. By using Corollary~\ref{l8}, Lemma~\ref{lem:app}, and Proposition~\ref{l3} as in \eqref{use 3}, we get
\begin{equation}\label{S_1 subseteq}
\begin{split}
&S_1:=\Gamma_1(T_0)\subseteq\Gamma_1(\emptyset)=(0,y_1^*]\cap\X,\\
&S_1 \cap \big((0,y^*_2]\cap\X \big)  = \Gamma_1(T_0)\cap T_0 = \emptyset.
\end{split}
\end{equation}
As $S_1\neq\emptyset$ and $y_2^*<y_1^*$, the above two inequalities imply $S_1 = A\cap\X$ for some nonempty $A\subseteq (y_2^*,y_1^*]$. With $X$ being a submartingale, the same argument in the second half of the proof of \cite[Lemma 5.1]{HZ19} can be repeated here, showing that the set $A$ has to be connected, i.e. $A=[c_1,c'_1]$ for some $c_1, c_1'\in\X$ with $y_2^*<c_1\le c'_1\le y_1^*$. Now, by the same argument in \cite[Lemma 5.3]{HZ19}, $c_1'\in\X$ needs to be large or equal to $\frac{1-\alpha_1^1}{u-\alpha_1^1}$, otherwise $S_1=[c_1,c'_1]\cap\X$ would not belong to $\cE_1^{T_0}$ (which would contradict Proposition~\ref{l3}). We then conclude $c'_1=y_1^*$ and thus $S_1 = [c_1,y^*_1] \cap\X$. Next, since $T_1:=\Gamma_2(S_1)\neq\emptyset$ and $T_1\subseteq T_0=(0,y_2^*]\cap\X$, we must have $T_1 = B\cap\X$ for some nonempty $B\subseteq (0,y_2^*]$. Again, by the same argument in the second half of the proof of \cite[Lemma 5.1]{HZ19}, $B$ is connected. Also, we must have $\inf B=0$, otherwise $T_1=B\cap\X\notin\cE_2^{S_1}$ (which would contradict Proposition~\ref{l3}). Thus, $T_1 = (0,d_1]\cap\X$ for some $d_1\in\X$ with $0<d_1\le y_2^*$. As $(S_n)$ is nondecreasing, $(T_n)$ is nonincreasing, and $T_n\cap S_n = \Gamma_2(S_n)\cap S_n =\emptyset$ (by Proposition~\ref{l3}), there exist nonincreasing sequences $(c_n)$ and $(d_n)$ of positive reals with $d_n< c_n$ such that
\begin{equation}\label{S_n, T_n}
S_n = [c_n,y^*_1] \cap\X\quad \hbox{and}\quad T_n = (0,d_n]\cap\X,\quad \forall n\in\N. 
\end{equation}
Now, if there exists $n^*\in\N$ such that $c_{n^*}=c_{n^*+1}$ or $d_{n^*-1}=d_{n^*}$, we get $\Gamma_2(S_{n^*})=T_{n^*}$ and $\Gamma_1(T_{n^*})=S_{n^*}$, so that $(S_\infty,T_\infty)=(S_{n^*},T_{n^*})\in\widehat\cE$ by Lemma~\ref{lem:sufficient}. If there exists no such $n^*\in\N$, the iterative procedure \eqref{alternating} continues indefinitely with $c_n\downarrow 0$ and $d_n\downarrow 0$, leading to $(S_\infty,T_\infty)=((0,y^*_1]\cap\X,\emptyset)$, which belongs to $\widehat\cE$ by Remark~\ref{rem:1>2}.
\end{proof}

Finally, let us explore the more extreme situation where one firm is highly impatient while the other is highly patient.  

\begin{corollary}\label{coro:1>2}
If $\beta_1>0$ is sufficiently large and $\beta_2>0$ is sufficiently small, 
the alternating iterative procedure \eqref{alternating} yields $(S_\infty,T_\infty)= ((0,y_1^*]\cap\X, \emptyset) \in \widehat\cE$.
\end{corollary}

\begin{proof}
Take $\beta_1>\overline\beta:= \frac{u}{u-1}\left(\frac{N}{K}-1\right)+\frac{1}{u-1}>0$. With $\beta_1$ fixed, in view of \eqref{alpha 1} and \eqref{X app}, there exists $\beta^*>0$ small enough such that for $\beta_2<\beta^*$, we have $\frac{1-\alpha_2^1}{u-\alpha_2^1}<\frac{1-\alpha_1^1}{u-\alpha_1^1}$ and the interval
\begin{equation}\label{two points}
\bigg(\frac{1-\alpha_2^1}{u-\alpha_2^1}K, \frac{1-\alpha_1^1}{u-\alpha_1^1}K\bigg)\ \hbox{contains at least two points in $\X$}. 
\end{equation}
Define $\underline{\beta}:=p\left(\frac{N}{K}-1\right)> 0$ and take $\beta_2<\min\{\beta^*, \underline\beta\}$. Note that $\beta_1>\overline\beta>\underline\beta> \beta_2$, where the second inequality is due to $p<1$. With $\beta_1>\beta_2$, we deduce from \eqref{alpha 1}, \eqref{y^*}, and \eqref{two points} that $y_1^*> y_2^*$ and the interval 
\begin{equation}\label{one point}
(y_2^*,y_1^*)\ \hbox{contains at least one point in $\X$}. %
\end{equation}

Following \eqref{alternating}, $S_0 :=\emptyset$ and $T_0 := \Gamma_2(S_0)= \Gamma_2(\emptyset) = (0,y_2^*]\cap\X$, where the last equality is due to Lemma~\ref{lem:app}. 
By \eqref{one point}, $u y_2^*\in\X$ must belong to $(y_2^*,y_1^*)$. 
Recall from \eqref{S_1 subseteq} that 
\begin{equation}\label{S_1 coro}
S_1\subseteq [u y_2^*,y_1^*]\cap\X.
\end{equation} 
For any $x\in [u y_2^*,y_1^*)\cap\X$, observe that
\begin{align}
V_1^{T_0}(x,\emptyset) = \sup_{1\le\tau\le\rho^+_\emptyset}\E_x[F_1(\tau,T_0)]\le \frac{N}{1+\beta_1} &< K-\frac{1}{u} K \notag\\
&< K-\frac{1-\alpha_1^1}{u-\alpha_1^1}K\le K-x= J_1(x,0,\rho_{T_0}),\label{for T}
\end{align}
where the second inequality follows from $\beta_1>\overline\beta$, the third inequality is due to $u>1$ and $0<\alpha_1^1<1$ (recall \eqref{alpha_i^n}), and the last inequality stems from $x\in\X$, $x<y^*_1$, and the definition of $y^*_1$ in \eqref{y^*}. This shows that $[u y_2^*,y_1^*)\cap\X\subseteq \Phi^{T_0}_1(\emptyset)$. Then, in view of \eqref{S_n}-\eqref{S_infty}, $S_1:=\Gamma_1(T_0)\supseteq \Phi^{T_0}_1(\emptyset) \supseteq [u y_2^*,y_1^*)\cap\X$. 
In particular, $S_1\neq\emptyset$ as it must contain $u y^*_2$. If $T_n=\emptyset$ for some $n>1$, Case 2 of Proposition~\ref{prop:1>2} immediately gives $(S_\infty,T_\infty)= ((0,y_1^*]\cap\X, \emptyset) \in \widehat\cE$, as desired. Hence, we assume $T_n\neq\emptyset$ for all $n\in\N$ in the rest of the proof. 

As argued below \eqref{S_1 subseteq}, $T_1:= \Gamma_2(S_1) = (0,d_1]\cap\X$ for some $d_1\in\X$ with $0<d_1\le y^*_2$. We claim that $d_1< y_2^*$. First, observe from \eqref{y^*} that $y_2^*< u\cdot \frac{1-\alpha_2^1}{u-\alpha_2^1} K< K$. Consider the function 
\[
\eta(x):= \frac{p(N-K)+x(1-\frac{1-p}{u})}{K-x}\quad \hbox{for $x\in [0,K)$.}
\]
As $\eta'(x)>0$ for all $x\in[0,K)$, we have $\underline\beta = p\left(\frac{N}{K}-1\right) = \eta(0) < \eta(x)$ for all $x\in(0,K)$. Recalling $\beta_2<\underline\beta$, we obtain $\beta_2 < \eta(x)$ for all $x\in(0,K)$. By direct calculation, this is equivalent to
\begin{equation}\label{iff eta}
(1-p) \frac{K-x/u}{1+\beta_2} +p \frac{N}{1+\beta_2}> K-x, \quad \forall x\in(0,K).
\end{equation}
Assume to the contrary that $d_1=y_2^*$. Then,
\[
J_2(d_1,\rho^+_{T_1},\rho_{S_1})=(1-p) \frac{K-d_1/u}{1+\beta_2} +p \frac{N}{1+\beta_2}> K-d_1 = J_2(d_1, 0, \rho_{S_1}),
\]
where the inequality follows from $d_1=y_2^*<K$ and \eqref{iff eta}. This indicates $T_1= (0,d_1]\cap\X \notin \cE_2^{S_0}$, a contradiction to $T_1= \Gamma_2(S_1)\in\cE_2^{S_0}$ (by Proposition~\ref{l3}). With $d_1<y_2^*$ established, we conclude $T_1\subsetneq T_0$. 
For any $x\in [u d_1,y^*_1)\cap\X$, the same calculation as in \eqref{for T} (with $T_0$ replaced by $T_1$) yields $V_1^{T_1}(x,\emptyset)<J_1(x,0,\rho_{T_1})$. By the same arguments below \eqref{for T}, this yields $S_2:=\Gamma_1(T_1)\supseteq \Phi^{T_1}_1(\emptyset) \supseteq [u d_1,y_1^*)\cap\X$; particularly, we have $u d_1\in S_2$. As $S_2\supseteq S_1$ by construction, we deduce from $u d_1\in S_2$, \eqref{S_1 coro}, and $d_1<y_2^*$, that  $S_2\supsetneq S_1$. 
By following the same arguments as above, we can show recursively that $T_n\subsetneq T_{n-1}$ and $S_{n+1}\supsetneq S_n$ for all $n>1$. By Case 3 of Proposition~\ref{prop:1>2}, $(S_\infty,T_\infty)= ((0,y_1^*]\cap\X, \emptyset) \in \widehat\cE$.
\end{proof}

Propositions~\ref{prop:1<2} and \ref{prop:1>2}, along with Corollary \ref{coro:1>2}, admit interesting economic interpretations. Intuitively, a firm can demonstrate a strong determination not to give in, so as to coerce the other firm into giving in in the negotiation. Whether this strategy will work depends on the impatience levels of the two firms. For the case where Firm 1 is less impatient than Firm 2 (i.e. $\beta_1\le \beta_2$), Propositions~\ref{prop:1<2} shows that when Firm 1 insists that it will never give in (i.e. $S_0=\emptyset$), Firm 2 indeed gives in by taking the stopping policy $T_0 = (0,y_2^*]\cap\X$, and $(S_0,T_0) = (\emptyset,(0,y_2^*]\cap\X)$ is already a sharp inter-personal equilibrium. For the case where Firm 1 is more impatient than Firm 2 (i.e. $\beta_1> \beta_2$), Propositions~\ref{prop:1>2} shows that the negotiation becomes more complicated and Firm 1's coercion does not necessarily work. In particular, if Firm 1 is sufficiently impatient and Firm 2 is sufficiently patient, 
Corollary \ref{coro:1>2} shows that Firm 1's coercion must fail, and fail in a drastic way---the coercer becomes coerced. While Firm 1 started with $S_0=\emptyset$, trying to coerce Firm 2 into giving in, the alternating game-theoretic reasoning \eqref{alternating} eventually lead to the sharp inter-personal equilibrium $((0,y_1^*]\cap\X, \emptyset)$. That is, it is Firm 1 who ultimately gives in by taking the stopping policy $(0,y_1^*]\cap\X$, while Firm 2 in the end decides not to stop at all.

\small
\bibliographystyle{siam}
\bibliography{refs}

\begin{thebibliography}{10}

\bibitem{AL11}
{\sc N.~Akakpo and C.~Lacour}, {\em Inhomogeneous and anisotropic conditional
  density estimation from dependent data}, Electron. J. Stat., 5 (2011),
  pp.~1618--1653.

\bibitem{BY17}
{\sc E.~Bayraktar and S.~Yao}, {\em On the robust {D}ynkin game}, Ann. Appl.
  Probab., 27 (2017), pp.~1702--1755.

\bibitem{BZZ19}
{\sc E.~Bayraktar, J.~Zhang, and Z.~Zhou}, {\em Time consistent stopping for
  the mean-standard deviation problem---the discrete time case}, SIAM J.
  Financial Math., 10 (2019), pp.~667--697.

\bibitem{https://doi.org/10.1111/mafi.12293}
\leavevmode\vrule height 2pt depth -1.6pt width 23pt, {\em Equilibrium concepts
  for time-inconsistent stopping problems in continuous time}, Mathematical
  Finance, 31 (2021), pp.~508--530.

\bibitem{Bismut77}
{\sc J.-M. Bismut}, {\em Sur un probl\`eme de {D}ynkin}, Z.
  Wahrscheinlichkeitstheorie und Verw. Gebiete, 39 (1977), pp.~31--53.

\bibitem{BKM17}
{\sc T.~Bj{\"o}rk, M.~Khapko, and A.~Murgoci}, {\em On time-inconsistent
  stochastic control in continuous time}, Finance and Stochastics, 21 (2017),
  pp.~331--360.

\bibitem{BMZ14}
{\sc T.~Bj{\"o}rk, A.~Murgoci, and X.~Y. Zhou}, {\em Mean-variance portfolio
  optimization with state-dependent risk aversion}, Math. Finance, 24 (2014),
  pp.~1--24.

\bibitem{CL18}
{\sc S.~Christensen and K.~Lindensj\"{o}}, {\em On finding equilibrium stopping
  times for time-inconsistent {M}arkovian problems}, SIAM J. Control Optim., 56
  (2018), pp.~4228--4255.

\bibitem{CL20}
\leavevmode\vrule height 2pt depth -1.6pt width 23pt, {\em On time-inconsistent
  stopping problems and mixed strategy stopping times}, Stochastic Process.
  Appl., 130 (2020), pp.~2886--2917.

\bibitem{CK96}
{\sc J.~Cvitani\'{c} and I.~Karatzas}, {\em Backward stochastic differential
  equations with reflection and {D}ynkin games}, Ann. Probab., 24 (1996),
  pp.~2024--2056.

\bibitem{DFM18}
{\sc T.~De~Angelis, G.~Ferrari, and J.~Moriarty}, {\em Nash equilibria of
  threshold type for two-player nonzero-sum games of stopping}, Ann. Appl.
  Probab., 28 (2018), pp.~112--147.

\bibitem{Dixit94}
{\sc A.~Dixit and R.~Pindyck}, {\em Investment under uncertainty}, Princeton
  University Press, Princeton, NJ,  (1994).

\bibitem{Dynkin69}
{\sc E.~B. Dynkin}, {\em Game variant of a problem on optimal stopping}, Soviet
  Math. Dokl., 10 (1969), pp.~270--274.

\bibitem{EWZ18}
{\sc S.~Ebert, W.~Wei, and X.~Y. Zhou}, {\em Weighted discounting---on group
  diversity, time-inconsistency, and consequences for investment}, Journal of
  Economic Theory, 189 (2020), p.~105089.

\bibitem{EL06}
{\sc I.~Ekeland and A.~Lazrak}, {\em Being serious about non-commitment:
  subgame perfect equilibrium in continuous time}, tech. rep., University of
  British Columbia, 2006.
\newblock Available at http://arxiv.org/abs/math/0604264.

\bibitem{EMP12}
{\sc I.~Ekeland, O.~Mbodji, and T.~A. Pirvu}, {\em Time-consistent portfolio
  management}, SIAM J. Financial Math., 3 (2012), pp.~1--32.

\bibitem{EP08}
{\sc I.~Ekeland and T.~A. Pirvu}, {\em Investment and consumption without
  commitment}, Math. Financ. Econ., 2 (2008), pp.~57--86.

\bibitem{F05}
{\sc E.~Z. Ferenstein}, {\em On randomized stopping games}, in Advances in
  dynamic games, vol.~7 of Ann. Internat. Soc. Dynam. Games, Birkh\"{a}user
  Boston, Boston, MA, 2005, pp.~223--233.

\bibitem{HZ09}
{\sc S.~Hamad\`ene and J.~Zhang}, {\em The continuous time nonzero-sum {D}ynkin
  game problem and application in game options}, SIAM J. Control Optim., 48
  (2009/10), pp.~3659--3669.

\bibitem{HN18}
{\sc Y.-J. Huang and A.~Nguyen-Huu}, {\em Time-consistent stopping under
  decreasing impatience}, Finance and Stochastics, 22 (2018), pp.~69--95.

\bibitem{HNZ20}
{\sc Y.-J. Huang, A.~Nguyen-Huu, and X.~Y. Zhou}, {\em General stopping
  behaviors of na{\"\i}ve and noncommitted sophisticated agents, with
  application to probability distortion}, Mathematical Finance, 30 (2020),
  pp.~310--340.

\bibitem{HW20}
{\sc Y.-J. Huang and Z.~Wang}, {\em Optimal equilibria for multi-dimensional
  time-inconsistent stopping problems}, SIAM J. Control Optim., 59 (2021),
  pp.~1705--1729.

\bibitem{HY19}
{\sc Y.-J. Huang and X.~Yu}, {\em Optimal stopping under model ambiguity: a
  time-consistent equilibrium approach}, Mathematical Finance, 31 (2021),
  pp.~979--1012.

\bibitem{HZ19}
{\sc Y.-J. Huang and Z.~Zhou}, {\em The optimal equilibrium for
  time-inconsistent stopping problems---the discrete-time case}, SIAM J.
  Control Optim., 57 (2019), pp.~590--609.

\bibitem{HZ20}
\leavevmode\vrule height 2pt depth -1.6pt width 23pt, {\em Optimal equilibria
  for time-inconsistent stopping problems in continuous time}, Mathematical
  Finance, 30 (2020), pp.~1103--1134.

\bibitem{Kosorok-book-08}
{\sc M.~R. Kosorok}, {\em Introduction to empirical processes and
  semiparametric inference}, Springer Series in Statistics, Springer, New York,
  2008.

\bibitem{LS13}
{\sc R.~Laraki and E.~Solan}, {\em Equilibrium in two-player non-zero-sum
  {D}ynkin games in continuous time}, Stochastics, 85 (2013), pp.~997--1014.

\bibitem{LM84}
{\sc J.-P. Lepeltier and M.~A. Maingueneau}, {\em Le jeu de {D}ynkin en
  th\'{e}orie g\'{e}n\'{e}rale sans l'hypoth\`ese de {M}okobodski},
  Stochastics, 13 (1984), pp.~25--44.

\bibitem{LT89}
{\sc G.~Loewenstein and R.~Thaler}, {\em Anomalies: Intertemporal choice},
  Journal of Economic Perspectives, 3 (1989), pp.~181--193.

\bibitem{MS86}
{\sc R.~McDonald and D.~Siegel}, {\em The value of waiting to invest}, The
  Quarterly Journal of Economics, 101 (1986), pp.~707--727.

\bibitem{Morimoto84}
{\sc H.~Morimoto}, {\em Dynkin games and martingale methods}, Stochastics, 13
  (1984), pp.~213--228.

\bibitem{Morimoto86}
\leavevmode\vrule height 2pt depth -1.6pt width 23pt, {\em Non-zero-sum
  discrete parameter stochastic games with stopping times}, Probab. Theory
  Relat. Fields, 72 (1986), pp.~155--160.

\bibitem{Nagai87}
{\sc H.~Nagai}, {\em Non-zero-sum stopping games of symmetric {M}arkov
  processes}, Probab. Theory Related Fields, 75 (1987), pp.~487--497.

\bibitem{Neveu-book-75}
{\sc J.~Neveu}, {\em Discrete-parameter martingales}, North-Holland Publishing
  Co., Amsterdam-Oxford; American Elsevier Publishing Co., Inc., New York,
  revised~ed., 1975.

\bibitem{Ohtsubo87}
{\sc Y.~Ohtsubo}, {\em A nonzero-sum extension of {D}ynkin's stopping problem},
  Math. Oper. Res., 12 (1987), pp.~277--296.

\bibitem{Pollak68}
{\sc R.~A. Pollak}, {\em Consistent planning}, The Review of Economic Studies,
  (1968), pp.~201--208.

\bibitem{RSV01}
{\sc D.~Rosenberg, E.~Solan, and N.~Vieille}, {\em Stopping games with
  randomized strategies}, Probab. Theory Related Fields, 119 (2001),
  pp.~433--451.

\bibitem{Sart14}
{\sc M.~Sart}, {\em Estimation of the transition density of a {M}arkov chain},
  Ann. Inst. Henri Poincar\'{e} Probab. Stat., 50 (2014), pp.~1028--1068.

\bibitem{SS04}
{\sc E.~Shmaya and E.~Solan}, {\em Two-player nonzero-sum stopping games in
  discrete time}, Ann. Probab., 32 (2004), pp.~2733--2764.

\bibitem{Smith95}
{\sc J.~E. Smith and R.~F. Nau}, {\em Valuing risky projects: option pricing
  theory and decision analysis}, Management Science, 41(5) (1995),
  pp.~795--816.

\bibitem{Strotz55}
{\sc R.~H. Strotz}, {\em Myopia and inconsistency in dynamic utility
  maximization}, The Review of Economic Studies, 23 (1955), pp.~165--180.

\bibitem{Thaler81}
{\sc R.~Thaler}, {\em Some empirical evidence on dynamic inconsistency}, Econ.
  Lett., 8 (1981), pp.~201--207.

\bibitem{TV02}
{\sc N.~Touzi and N.~Vieille}, {\em Continuous-time {D}ynkin games with mixed
  strategies}, SIAM J. Control Optim., 41 (2002), pp.~1073--1088.

\bibitem{Yasuda85}
{\sc M.~Yasuda}, {\em On a randomized strategy in {N}eveu's stopping problem},
  Stochastic Process. Appl., 21 (1985), pp.~159--166.

\bibitem{Yong12}
{\sc J.~Yong}, {\em Time-inconsistent optimal control problems and the
  equilibrium \textsc{HJB} equation}, Mathematical Control and Related Fields,
  3 (2012), pp.~271--329.

\end{thebibliography}

\end{document}